\documentclass{article}
\usepackage[utf8]{inputenc}
\usepackage{amssymb, amsmath, amsthm, authblk}
\usepackage{amsfonts, bbm}
\usepackage{color, graphicx, caption, subcaption}
\usepackage[usenames,dvipsnames,svgnames,table]{xcolor}
\usepackage{tikz, multicol,hyperref}
\usepackage[margin=1in,marginparwidth=0.8in, marginparsep=0.1in]{geometry}
\usetikzlibrary{matrix}
\usepackage{tikz-cd}
\usepackage{float}

\usepackage{cleveref}

\newcommand\numberthis{\refstepcounter{equation}\tag{\theequation}}

\newcommand{\Aff}{{\mathbb A}}

\newcommand{\C}{{\mathbb C}}

\newcommand{\Q}{{\mathbb Q}}
\newcommand{\R}{{\mathbb R}}
\newcommand{\Z}{{\mathbb Z}}
\newcommand{\N}{{\mathbb N}}

\newcommand{\calA}{{\mathcal A}}

\newcommand{\calC}{{\mathcal C}}

\newcommand{\calF}{{\mathcal F}}
\newcommand{\calG}{{\mathcal G}}

\newcommand{\calI}{{\mathcal I}}

\newcommand{\calM}{{\mathcal M}}
\newcommand{\calN}{{\mathcal N}}
\newcommand{\calO}{{\mathcal O}}
\newcommand{\calP}{{\mathcal P}}

\newcommand{\calR}{{\mathcal R}}

\newcommand{\calT}{{\mathcal T}}

\newcommand{\calX}{{\mathcal X}}

\newcommand{\frakg}{\mathfrak{g}}
\newcommand{\frakd}{\mathfrak{d}}

\newcommand{\frakD}{\mathfrak{D}}

\newcommand{\frakm}{\mathfrak{m}}
\newcommand{\frakj}{\mathfrak{j}}
\newcommand{\fD}{\mathfrak{D}}

\newcommand{\fp}{\mathfrak{p}}
\newcommand{\frakp}{\mathfrak{p}}

\DeclareMathOperator{\Gr}{Gr}

\DeclareMathOperator{\Hom}{Hom}

\newcommand{\cA}{\mathcal{A}}
\newcommand{\cAp}{{\mathcal A_{\text{prin}}}}

\newcommand{\St}{\operatorname{St}}

\newcommand{\Var}{\operatorname{Var}}
\newcommand{\rep}{\operatorname{rep}}
\newcommand{\End}{\operatorname{End}}
\newcommand{\GL}{\operatorname{GL}}

\newcommand{\im}{\operatorname{im}}
\newcommand{\supp}{\operatorname{supp}}
\newcommand{\Ext}{\operatorname{Ext}}

\newcommand{\Spec}{\operatorname{Spec}}

\newcommand{\Mono}{\operatorname{Mono}}

\newcommand{\Rep}{\operatorname{Rep}}
\newcommand{\modu}{\operatorname{mod}}
\newcommand{\Irr}{\operatorname{Irr}}
\newcommand{\op}{\operatorname{op}}

\newcommand{\raw}{{\rightarrow}}
\newcommand{\rraws}{{\rightrightarrows}}
\newcommand{\enpt}{{\mathcal{Q}}}  
\newcommand{\stab}{{\theta}}     

\newcommand{\seed}{\textbf{s}}
\newcommand{\wM}{\widetilde{M}}
\newcommand{\wN}{\widetilde{N}}
\newcommand{\wQ}{\widetilde{Q}}
\newcommand{\wpp}{\widetilde{p}}
\newcommand{\wseed}{\widetilde{\textbf{s}}}
\newcommand{\n}{\textsl{n}}
\newcommand{\Dprin}{\frakD_{\seed}^{\calA_{\text{prin}}}}
\newcommand{\calE}{{\textsf{g}}}
\newcommand{\prin}{{\text{prin}}}
\newcommand{\g}{{\text{g}}}

\newcommand{\bc}{{\textbf{c}}}

\newtheorem{theorem}{Theorem}[section]

\newtheorem{lemma}[theorem]{Lemma}
\newtheorem{prop}[theorem]{Proposition}

\newtheorem{definition}[theorem]{Definition}
\newtheorem{example}[theorem]{Example}

\newtheorem{remark}[theorem]{Remark}

\title{Theta functions and quiver Grassmannians}
\author{Man-Wai Cheung}
\date{}

\begin{document}

\newcommand{\Addresses}{{
  \bigskip
  \footnotesize
  
  \textsc{Man Wai Cheung, Department of Mathematics, One Oxford Street Cambridge, Harvard University, MA 02138}\par\nopagebreak
  \textit{E-mail address:} \texttt{mwcheung@math.harvard.edu}
}}

\maketitle

\begin{abstract}
  In this article, we use the relationship between cluster scattering diagrams and stability scattering diagrams to relate quiver representations with these diagrams.
With a notion of positive crossing of a path $\gamma$, we show that if $\gamma$ has positive crossing in the scattering diagram, then it goes in the opposite direction of the Auslander-Reiten quiver of $Q$. 
We then give the Hall algebra theta functions which recover the cluster character formula by the Euler characteristic map.
At last, we define the Hall algebra broken lines and then are able to give the stratification of the quiver Grassmannians by the bending of the broken lines.
\end{abstract}

\section{Introduction}

Fomin and Zelekinsky set up the theory of cluster algebras in 2000 in order to understand total positivity in algebraic groups and canonical bases in quantum groups.
 Roughly speaking, a cluster algebra is a subring of a field of rational functions.
 To define a cluster algebra,
 instead of knowing all the generators at the beginning,
 we start with initial data called an \textit{initial seed} which includes some cluster variables.
 There is a procedure called mutation to generate more seeds, and
 a cluster algebra is defined to be the subring generated by the cluster variables in all the seeds.
 
The theory developed rapidly and connected with many other areas, e.g., Poisson geometry, integrable systems, higher Teichm\"uller spaces, algebraic geometry, and quiver representations.
 Later Fock and Goncharov introduced a geometric point of view in \cite{FG}.
 They introduce the $\calA$ and $\calX$ cluster varieties which can be obtained by gluing `seed tori' by birational map called cluster transformations.
 
In another world in mathematics, namely mirror symmetry, the concepts of
scattering diagram, theta functions and broken lines were introduced to obtain
a deeper understanding of the dualities involved.
 Two-dimensional scattering diagrams were introduced by Kontsevich and Soibelman in \cite{KS} to study K3 surfaces.
 Later Gross and Siebert \cite{GS} considered general scattering diagrams
 to describe toric degenerations of Calabi-Yau varieties in order to construct mirror pairs.
 On the other hand,
 the notion of broken lines was developed by Gross in \cite{grossp2} to understand Landau-Ginzburg mirror symmetry for ${\mathbb P}^2$.
 Then Siebert, Carl and Pumperla \cite{CPS} made use of broken lines to describe regular functions in the context of \cite{GS}, and in particular to construct Landau-Ginzburg mirrors to varieties with effective anti-canonical bundle.
 In order to construct mirrors to log Calabi-Yau surfaces with maximal boundary with similar ideas,
 theta functions were introduced by Gross, Hacking and Keel in \cite{GHKlog}.
An earlier suggestion of Abouzaid, Gross and Siebert of calculating products
using tropical Morse trees was made explicit in \cite{GHKlog} with a formula 
for multiplication of theta functions in terms of trees of broken lines.
 
The discoveries of \cite{GHK_bir} and \cite{ghkk} revealed that there is a strong connection between cluster algebras and scattering diagrams. 
 In particular, we can view each chamber in the scattering diagram as one of the seeds in the cluster algebras.
 The chambers can then be viewed as corresponding to the `seed tori' which are `glued' along the walls in the scattering diagrams.
 Theta functions can then be viewed as cluster variables.
 The set of theta functions are proposed to be a canonical basis for cluster algebras.
 There are many other proposed bases as well, e.g. \cite{geiss2012generic}, \cite{keller2014cluster}, \cite{dupont2013atomic}, \cite{MSW}.
 Together with Gross, Muller, Musiker, Rupel, Stella and Williams \cite{cheung2015greedy}, we have proved that the theta basis in rank 2 agrees with the greedy basis. 
 
A natural question to ask is if there is any deep, intrinsic relation between the two subjects.
First we can relate the scattering diagrams with the Auslander-Reiten quivers in representation theory. 
Given a skew symmetric type cluster algebra, one can consider the quiver $Q$ associated to the algebras. 
Then we can consider the stability scatterings diagram introduced by Bridgeland \cite{Bridge} (stated in Theorem \ref{thm:Hallscattering}) and the corresponding Auslander-Reiten quiver of $Q$ as defined in Definition \ref{def:arquiver}.

\begin{theorem}[Theorem \ref{thm:arquiver}] \label{thm:ar}
Let $\gamma$ be a path crossing the outgoing piece of $\frakd_1 \in \frakD$ and then $\frakd_2 \in \frakD$ with both crossings positive.
  Let $C_i$ be the corresponding indecomposable representation for $\frakd_i$, $i=1,2$.
  Then we have $C_2$ is a predecessor of $C_1$ in the Auslander-Reiten quiver of $Q$.
\end{theorem}

 With insight from Caldero and Chapoton in \cite{ccformula}, theta functions (defined in \eqref{eqn:thetaf}) are related to quiver Grassmannians.
 The definition of theta functions rely on the piecewise linear paths with attached monomials data, called broken lines (Definition \ref{brokendef}), in the scattering diagrams.
Thus, we propose each broken line corresponds to a family of subrepresentations of an ambient quiver representation.
 More specifically, the initial slope of the broken line tells us which ambient representation $D$ we are working with.
 On the other hand, the final slope of the broken line gives one of the subrepresentations $E$ of $D$.
 So we would like to investigate the meaning of bendings from the initial to the finial slopes of the broken lines.

By the machinery of motivic Hall algebras developed by Joyce \cite{joyce2007configurations} and the stability scatterings diagram in \cite{Bridge},
 we can describe the usual wall crossing via conjugation by Hall algebra elements.
In Section \ref{sec:halltheta}, we are able to give the Hall algebra theta functions

 \begin{theorem} \label{thm:theta} [Theorem \ref{thm:sectheta}]
Fix a stability scattering diagram $\frakD$.
Let $\enpt$ be a point in the positive chamber of $\frakD$.
Consider a point $\textbf{m}_0$ in the cluster complex.
Then we can write $\textbf{m}_0 = -\calE (\textbf{d})$ for some $\textbf{d} \in (\Z_{\geq 0})^n$, where $\calE$ defined in \eqref{eq:gg} and we are going to see it is the $g$ vector. Then the Hall algebra theta function is
\[
\vartheta_{(\textbf{m}_0,0)} = \chi (\calG_{\calF (m_0)} (D))  z^{(\textbf{m}_0,0)},
\]
where objects in $\calG_{\calF (m_0)} (D)$ are representations $C$ in $\calF (m_0)$ equipped with an inclusion into $D$.
 By applying the Euler characteristic map $\chi$ (defined in \eqref{eqn:chi1}),
we get
\[
	\vartheta_{(\textbf{m}_0,0)} 
	 = z^{(- \calE (D),0)} \sum_{0 \leq \bc \leq \textbf{d}} \chi (\Gr(c,D)) z^{(p^*(\textbf{\bc}), \textbf{\bc})}.
\]
which is the cluster character formula.
 \end{theorem}
 
 Next, we find that each bending of a broken line tells us about stratas of quiver Grassmanians.
The bendings of the broken lines actually describe the filtration of the quiver subrepresentation $C$.
 We will then be able to give a stratification of the quiver Grassmannian.
 Note that the stratifications given by the broken lines differ from those appear in other literature, e.g. \cite{CZformula}.
 
 Consider an indecomposable quiver representation $D$ which is not regular.
 Then $m = -\calE (d)$ lies in the cluster complex.
 Now consider a broken line $\gamma : (\infty  ,0 ] \raw M_{\R} \setminus \{ 0 \}$ with endpoint $\enpt$ in the positive chamber and initial slope $-\calE(\textbf{d})$.
 Let the final slope of the broken line be  $ -\calE(\textbf{d}) + p^*(\bc)$ for some $\bc$.
 Let $\gamma$ bend over the walls $\frakd_1, \dots , \frakd_s$ in the cluster complex and assume all the bendings are positive.
 Let $\lambda_1, \dots, \lambda_s$ be the degree of each bending.
 Denote by $\bc_i \in N^+$ the normal vectors of $\frakd_i$ for $i = 1, \dots s$. First, bending over $\frakd_1$ leads us
 
 \begin{theorem}[Theorem \ref{thm:firstbending}]
 Consider the first bending of $\gamma$ over a general point $\stab_1$ on the wall $\frakd_1$ which corresponds to pre-projective/ pre-injective indecomposable representation $C_1$ of dimension vector $\textbf{c}_1$.
 Then the Hall algebra wall crossing automorphism is
\[ \Phi_{\frakD} (\frakd_1) z^{-(\calE (d),0)}= \sum_{\lambda} \calG^1_{\lambda_1,\bc_1}(D) z^{-(\calE (d),0)},\]
where $\calG^1_{\lambda,\bc_1} (D) = \{ \psi : V \raw D | V $ is a representation of dimension vector $\lambda\bc_1$, and $\ker \psi$ contains no subrepresentation of dimension vector proportional to $\textbf{c}_1 \}$.
Let $V_1 = C_1^{\oplus \lambda_1}$. Then we have
\[
0 \subset V_1
\]
as a filtration obtained after the first bending.
\end{theorem}
 
In Section \ref{sec:2bending} and Section \ref{sec:kbend}, we fineally
explain how to understand bending inductively: 

\begin{theorem} [Theorem \ref{thm:recapsecondbend} for the second bending, Theorem \ref{thm:kbend}] \label{thm:morebending} 
Assume now $\gamma$ is taking the $j$-th bending from $\frakd_{j-1}^+$ to $\frakd_j$.
From the $j-1$-th bending, we obtain the filtration
\[
 0 \subset V_1 \subset \cdots \subset V_{j-1},
 \]
 with the quotient $V_{\ell}/ V_{\ell-1} = C_{\ell-1}^{\oplus \lambda_{\ell-1}}$.
Then the Hall algebra monomial associated to the $j$-th linear piece of the broken line would be
\[[\calG^{j-1}_{\lambda_1 , \cdots , \lambda_{j-1},\bc_{j-1}} \raw \calM] z^{-(\calE (d),0)},\]
where $[\calG^{j-1}_{\lambda_1 , \cdots , \lambda_{j-1},\bc_{j-1}} \raw \calM]$ 
is space having the Poincar\'{e} polynomial as
 \[ [\Aff^{\lambda_j \Ext^1 (V_{j-1}, C_j)} \times
 \Gr(\lambda_j, \Hom(C_j, D/ V_{j-1} )   - \Ext^1 (V_{j-1}, C_j))].
 \]
 At the last bending, 
 we obtain filtration for a subrepresentation $C$ of $D$ with dimension vector $\textbf{c}$.
\begin{equation*} 
	 0 \subset V_1 \subset V_2 \subset \cdots \subset V_s =E, 
\end{equation*}
where $V_j / V_{j-1} = C_j^{\oplus \lambda_j}$.
\end{theorem}
 
 Our calculation for each strata involves not only the tools from the theory of Hall algebras but also a closer understanding of quiver representations.
 More specifically, we need to understand the $\Hom$ and $\Ext$ groups between indecomposable quiver representations.

\subsection*{Acknowledgements}
The author would like to take this opportunity to thank my advisor, Mark Gross, for his endless patience, support, encouragement, and guidance during my years of PhD study.

\section{Introduction to Cluster algebras and quiver representations}
\subsection{Cluster algebras} \label{sec:cluster}

In this section, we will give the basic definition of cluster algebras. Since this article focuses on quivers, we will restrict ourselves to the skew-symmetric type. 
For the sake of consistency throughout the article, we will adopt the definition from \cite{FG} and \cite{GHK_bir}. The exchange matrix $\epsilon$ (defined below) differs by a transpose from the general definition the cluster community as stated in \cite{cluster1} where the exchange matrix is denoted as $B$.
We will first state the definition for the case without frozen variables and then we will mention the case with principal coefficients.

A cluster algebra $\calA$ of rank $n$ is a subring of the ring of rational functions in $n$ variables over a field $\Bbbk$, where $\Bbbk$ is a field of characteristic zero.
In this article, we will restrict to the skew-symmetric type cluster algebras without frozen variables. 
To define a cluster algebra, we will start by defining the seed. 

\begin{definition} \label{def:cluster}
A \emph{seed} is a pair $\Sigma = (\seed, \epsilon) $, where
\begin{itemize}
	\item $\seed$ is a set of algebraically independent elements $\{A_1, \dots A_\n \} $ lying in the field of rational functions of $n$ independent variables;
	\item $\epsilon=(\epsilon_{ij})$ is a skew-symmetric integer matrix.
\end{itemize}
Elements in $\seed$ are called \emph{cluster variables} and $B$ is called the \emph{exchange matrix}. 
The seed to start with is called the \emph{initial seed}. 
\end{definition}
To generate new seeds, there is a procedure called \emph{mutation} to obtain a new seed 
\[\mu_k ((\seed, \epsilon))  = \{ \{ A_1, \dots, \hat{A_k}, \dots, A_\n \} \cup \{A_k'\}, \mu_k (\epsilon) \}\] for $k = 1, \dots, n$, where 
\[ A_k A_k ' = \prod_{\epsilon_{kj} >0 } A_j^{\epsilon_{kj}} + \prod_{\epsilon_{kj} <0 } A_i^{-\epsilon_{kj}}, \]
and
\[(\mu_k(\epsilon))_{i,j}
 =  \left\{
 \begin{array}{lr}
 -\epsilon_{ij} & \text{if}~ k = i~\text{or}~ k=j; \\ 
 \epsilon_{ij} + \frac{|\epsilon_{ki}| \epsilon_{jk} + \epsilon_{ki}| \epsilon_{jk} |}{2} & \text{otherwise.}
\end{array} \right.
\]
Starting from an initial seed, one can apply all possible sequences of compositions of mutations (possibly infinite). This gives a set (possibly infinite) of all cluster variables generated under mutations. 
The \emph{cluster algebra} $\calA = \calA(\epsilon, \seed)$ is the subalgebra of $\Bbbk (A_1, \dots , A_\n)$ generated by all cluster variables. 
A monomial in cluster variables all of which belong to one extended cluster will be called a \textbf{cluster monomial}

Every cluster algebra belongs to a series $\calA(\epsilon)$ consisting of all cluster algebras of the form $\calA (\epsilon, \seed)$, where $\epsilon$ is fixed and $\seed$ is allowed to vary. Two series $\calA (\epsilon)$ and $\calA(\epsilon')$ are \emph{strongly isomorphic} if $\epsilon$ and $\epsilon'$ can be obtained from each other by a sequence of mutation, modulo simultaneous relabeling of rows and columns.

Let us illustrated the construction by an example.

\begin{example}[Rank 2 cluster algebra] \label{ex:2alg}
In the rank 2 case, the initial seed can be represented as
\[
	\Sigma = ( \{ A_1, A_2 \}, \epsilon = 
	\left(
	\begin{array}{c c}
	0 & b \\ -b & 0
	\end{array}
	\right) ),
\]
where $b$ is a positive integer.
Mutating at 1 gives us
\[A_1 A_1' = 1+ A_2^b . \]
Now we denote $\vartheta_3 = A_1 '$ and $\vartheta_1=A_1$, $\vartheta_2 = A_2$. Then by repeating the mutation process and naming new cluster variables as $\vartheta_{k+1}$, we will get a recursive relation indexed by $k \in \Z$ as
\begin{equation}\label{eq:exchangerelations}
\vartheta_{k-1}\vartheta_{k+1}= \vartheta_k^b+1 
\end{equation}
The \emph{cluster algebra} $\cA(b)$ is the $\Z$-subalgebra of $\Q(A_1,A_2)$ which they
generate.  
Note that all the $\vartheta_k$, $k \neq 1, 2$ lies in $\Q(A_1,A_2)$. 
Each pair $\{\vartheta_k, \vartheta_{k+1}\}$ is called a \emph{cluster} and a monomial in the variables of a cluster is called a \emph{cluster monomial}.
\end{example}

In the above example, if $b =1$, then there are only 5 cluster variables in $\calA (1)$. 
In general, a cluster algebra with a finite number of cluster variables is called a 
\emph{cluster algebra of finite type}.

Fomin-Zelevinsky \cite{cluster2} showed that there is a bijection between the Cartan matrices $C$ for Dynkin diagram with cluster algebra $\calA (\epsilon)$, where $C(\epsilon) =C$ is defined as 
\[
c_{ij} = \left\{
 \begin{array}{ll}
 2 & \text{if}~ i =j; \\ 
 -|\epsilon_{ij}| & \text{if}~ i \neq j. 
\end{array} \right.
\]

\begin{theorem} \cite[Theorem 1.4]{cluster2} \label{thm:finite}
All cluster algebras in any series $\calA(\epsilon, -)$ are simultaneously of finite or infinite type. There is a canonical bijection between the Cartan matrices of finite type and the strong isomorphism classes of series of cluster algebras of finite type.
Under this bijection, a Cartan matrix $C$ of finite type corresponds to the series $\calA(\epsilon, -)$, where $\epsilon$ is an arbitrary skew-symmetric matrix with $C(\epsilon) = C$.
\end{theorem}

Next we will introduce the cluster algebras with principal coefficients $\cAp$. It is a cluster algebras with $n$ extra variables $X_1, \dots, X_n$, i.e. $\calA \subset \Bbbk (A_1, \dots , A_\n, X_1, \dots , X_n)$. We will denote $A_{n+i}$ as $X_i$ for $i = 1, \dots, k$. The mutation would only be carried on the variables $A_i$, i.e. the first $n$ variables. The exchange matrix $\widetilde{\epsilon}$ for $\cAp$ would then be 
\[
\widetilde{\epsilon} = 
	\left(
	\begin{array}{c c}
	\epsilon & I \\ -I & 0
	\end{array}
	\right) ),
\]
where $I$ is the $n \times n$ identity matrix. 

For any cluster variables in the cluster algebras, Fomin-Zelevinsky showed the Laurent phenomenon for the cluster variables. The following statement is stated for $\calA$ and $\cAp$ but the Laurent phenomenon holds for more general cluster algebras.
\begin{theorem} \cite{cluster1} \cite[Theorem 3.3.1]{introcl123}
For a cluster algebra without frozen variables, or with principal coefficients, each cluster variable can be expressed as a Laurent polynomial with integer coefficients in the elements of any cluster. \end{theorem}

Thus, a cluster variable $\vartheta$ would be a Laurent polynomial in $A_1, \dots A_\n, X_1, \dots, X_\n$. 
We can define the $F$-polynomial by specifying all $A_i=1$, i.e. 
\begin{equation} \label{eq:F}
    F (X_1, \dots, X_\n) := \vartheta_i (A_1 = \cdots = A_\n =1). 
\end{equation}

In particular, by \cite{cluster4}, there is a factorization of the cluster variable $\vartheta$ in terms of the $F$-polynomial:
\begin{equation} \label{eq:gvector}
\theta = \prod_j A_j^{g_j} \cdot F\big( X_1\prod_j A_j^{\epsilon_{1,j}}, \dots ,X_n\prod_j A_j^{\epsilon_{n,j}}\big)|_{X_1 = \cdots = X_n =1}.
\end{equation}

We will call the vector $\textbf{g} = (g_1, \dots, g_n)$ as the $g$-vector. In \cite{cluster4}, one can mutate the $X$ variables and obtain similar set of vectors for the $X$ variables. One should refer to \cite{cluster4} for the details. We will define the set of $c$ vectors as the set of row vectors in the $(1, n) \times (n+1 , 2n) $ part of the matrix $\epsilon$, for all $\epsilon$ obtained from the mutation process. 

Now fix a seed $\seed$ and let $\epsilon$ be the corresponding exchange matrix in the seed. 
Let $C_\seed$ be the $n \times n$ matrix obtained by taking transpose of the $(1, n) \times (n+1 , 2n) $ of $\epsilon$. 
Let $\vartheta_1, \dots, \vartheta_n$ be the cluster variables in $\seed$ and let $\textbf{c}_1, \dots , \textbf{c}_n$ be the corresponding $c$-vectors. 

Combining the sign-coherence property of $c$-vector (each vector $\textbf{c}_i$ has neither all entries nonnegative or all entries nonpositive) proved in \cite{ghkk}, 
Nakanishi-Zelevinsky showed the following duality between the $C$ and $G$ matrix. We will state the theorem for the skew-symmetric case:
\begin{theorem} \cite[Theorem 1.2]{nakanishi2012tropical} \label{thm:tropduality}
\begin{equation} \label{eq:cgduality}
    G_\seed^T = C_\seed^{-1}, 
\end{equation}
\end{theorem}

Note that this theorem implies the set of column vectors in $ G_\seed$ and $ C_\seed$ formed a dual basis. Nakanish-Zelevinsky named the above duality as tropical duality.

\subsection{Quiver Representations}
\subsubsection{Introduction to quiver representations}

A \emph{quiver} $Q$ of rank $n$ is a directed graph of $n$ vertices. 
More precisely $Q = (Q_0, Q_1, s,t)$, where 
$Q_0$ is the set of vertices of $Q$, $Q_1$ is the set of arrows of $Q$, $s, t : Q_1 \raw Q_0$ two maps. For an arrow $\alpha \in Q_1$, $s (\alpha)$ is the starting point of $\alpha$ and $ t(\alpha)$ is the end point of $\alpha$. 
It can be written as $\alpha: s(\alpha) \raw t (\alpha)$.
Let us assume the vertices of $Q$ are numbered $\{ 1, \dots, n \}$.
We will mainly talk about acyclic quivers $Q$. Thus we can further assume $ a < b$ if there is an arrow goes from vertex $a$ to $b$.

Note that one can associate a cluster algebra to a given quiver $Q $ by considering the rank of the cluster algebra as the rank of $Q$ and defining the exchange matrix $\epsilon$ in the initial seed as
\begin{equation} \label{eq:quivermat}
    \epsilon_{ij} = \text{number of arrows from } j \text{ to } i -    \text{number of arrows from } i \text{ to } j.
\end{equation}

One can see that for the case of cluster algebra with principal coefficients, we are associating a vertex $i'$ (called the frozen vertex)
to each vertex $i$ of the original quiver and there is an arrow from $i$ to $i'$.  

We would then consider the representation of the opposite quiver associated to \eqref{eq:quivermat}.
The opposite quiver $Q^{\op}$ of $Q$ is defined by reserving all the vertices of $Q$ and the arrows in $Q^{\op}$ are in opposite direction of the arows in $Q$. 
Note that there is an equivalence between the category of representations of the opposite quiver $\rep(Q^{\text{op}})$ (defined below) and the path algebra of $Q$ (as right module). 
Thus we will consider the representation of $Q^{op}$.
For the sake of simplicity, we will still denote the opposite quiver $Q$ when we discuss about quiver representation. 
Just be aware that we are actually considering the opposite quivers of the quivers associated to \eqref{eq:quivermat}.

\begin{definition} 
A $\C$-linear representation $V = (V_i, V_{\alpha} ) _{i \in Q_0, \alpha \in Q_1}$ of $Q$ is defined by:
\begin{itemize}
\item To each point $i$ in $Q_0$ is associated a $\C$ vector space $V_i$
\item To each arrow $\alpha : i \raw j$ in $Q_1$ is associated a $\C$-linear map $V_{\alpha} : V_i \raw V_j$.
\end{itemize}
It is called \emph{finite dimensional} if each vector space $V_a$ is finite dimension. 
In this case, the \emph{dimension vector} of $M$ is defined to be the vector
\[ \dim (V) = ( \dim V_i)_{i \in Q_0}. \]
\end{definition}

A \emph{morphism} of representations $\phi : V \raw W$ is a collection $\phi  = (\phi_a) _{a \in Q_0}$ of linear maps $\phi_a : V_a \raw W_a$ for each vertex $i$ such that the following commutative diagram commutes:
\[
\begin{tikzcd}
V_{s (\alpha)} \arrow[r, "\phi_{s (\alpha)}"] \arrow[d, "V_{\alpha}"] & W_{s (\alpha)} \arrow[d, "W_{\alpha}"] \\
V_{t (\alpha)} \arrow[r, "\phi_{t (\alpha)}"] & W_{t (\alpha)}
\end{tikzcd}
\]
Let $f: V \raw V'$ and $g: V' \raw V''$ be two morphisms of representations of $Q$. Then their composition is defined as $(gf)_a = g_a \circ f_a$ for $a \in Q_0$. Then $gf: V \raw V''$ is also a morphism of representations. 
This defines a category $\Rep(Q)$ of representations of $Q$. We denote by $\rep(Q)$ the full subcategory of $\Rep (Q)$ consisting of the finite dimensional representations.

Let $V = (V_a, V_{\alpha})$ and $W = (W_a, W_{\alpha})$ be representation of $Q$. Then the \emph{direct sum} $V \oplus W$ is defined as
\[ ( V \oplus W)_a = V_a \oplus W_a \]
\[ ( V \oplus W)_{\alpha} =
\left( \begin{array}{c c}
V_{\alpha} & 0 \\ 0 & W_{\alpha}
\end{array} \right)
\]

A representation $V$ is called \emph{indecomposable} if $V \neq 0$ and if $V = S \oplus T$, then $S = 0$ or $T =0$. 

Let $N = \Z Q_0$ and $M = \Hom(N, \Z)$.
Define a bilinear form, $\chi ( \cdot, \cdot)$, the Euler form on $N$ as 
\begin{equation} \label{def:euler}
\chi (\textbf{c}, \textbf{d}) = \sum_{i \in Q_0} c_i d_i - \sum_{\alpha: i \raw j} c_i d_j
\end{equation}
for $\textbf{d} \in N$. 
We further define a map $\calE: N \raw M$ by
\begin{equation} \label{eqn:calE}
\calE(\textbf{d}) = \chi ( \cdot, \textbf{d}).
\end{equation}

\begin{example} \label{ex:2quiver}
Let us take the Kronecker 2-quiver $Q_2$ as an example.
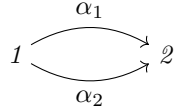
\begin{center}
\begin{tikzpicture}
\node (A) at (-1,0) {1};
\node (B) at (1,0) {2};
\path[->]
(A) edge [bend left=30, "$\alpha_1$"]  (B)
(A) edge [bend right=30, "$\alpha_2$"']  (B);
\end{tikzpicture}
\end{center}

Then the set of indecomposable representations are of the form
\[
\begin{tikzcd}
\C^n \arrow[r, yshift=0.7ex, "f^p_1"] \arrow[r, yshift=-0.7ex, "f^p_2"']
& \C^{n+1}
\end{tikzcd}
,
\begin{tikzcd}
\C^k \arrow[r, yshift=0.7ex, "f^r_1"] \arrow[r, yshift=-0.7ex, "f^r_2"']
& \C^{k}
\end{tikzcd}
,
\begin{tikzcd}
\C^{n+1} \arrow[r, yshift=0.7ex, "f^i_1"] \arrow[r, yshift=-0.7ex, "f^i_2"']
& \C^{n}
\end{tikzcd},
\]
where
\begin{align*}
    f^p_1 :& (x_1, \dots, x_n ) \mapsto (x_1, \dots, x_n , 0), \\
    f^p_2 :& (x_1, \dots, x_n ) \mapsto (0, x_1, \dots, x_n ), \\
    f^r_1 &= \mu (1, \dots, 1), \quad \text{for some } \mu, \\
    f^r_2 &= \left(
	\begin{array}{c c c c c}
	1 & \lambda & 0 & \cdot & 0 \\
	0 & 1      & \lambda & 0 \cdots \\
	& & \cdots & & \\
	0 & \dots & 0& 1 &\lambda  \\
	0 & \dots & 0& 0 &1 \\
	\end{array}
	\right), \quad \text{for some } \lambda, \\
	f^i_1:& (x_1, \dots, x_n ) \mapsto (x_1, \dots, x_{n-1}), \\
	f^i_2:& (x_1, \dots, x_n ) \mapsto (x_2, \dots, x_n ),
\end{align*}
for $n \in \Z$, $k \in \Z_{\geq 1}$. Further
\begin{align*}
    \calE \left( (n, n+1)\right) & = (2-n, n+1),\\
    \calE \left( (k,k)\right) & = (-k, k),\\
    \calE \left( (n+1, n)\right) & = (1-n, n).
\end{align*}
\end{example}

\subsubsection{Simple, projective and injective representations}
In this section, we can are to give explicit descriptions of the simple, indecomposable projective, indecomposable injective representations. 

\begin{definition}
A \emph{simple} representation is defined to be a non zero representation with no proper subrepresentations.
\end{definition}

Given a vertex $i$, define $S(i)$ to be the representation
\[
S(i)_j = 
\left\{ \begin{array}{rcl}
\C & \mbox{if} & j=i \\ 
0 & \mbox{if} & j \neq i  
\end{array}
\right.
\]
and $S(i) _{\alpha} = 0$.
This is the simple representation associated to the vertex $i$.
The collections $\{ S(i) \}$ are the set of all simple representations of $Q$.

\begin{definition}
Let $i$ be a vertex of Q. 
\begin{enumerate}
	\item Define the \emph{projective representation} $P(i)$ at vertex $i$  as
\[P(i) = (P(i)_j, P(i)_{\alpha})_{j \in Q_0, \alpha \in Q_1} \]
where $P(i)_j$ is the $\C$-vector space with basis the set of all paths from $i$ to $j$ in $Q$. For $\alpha : j \raw l$, $P(i)_{\alpha} : P(i)_j \raw P(i)_l$ is the linear map defined on the basis by composing the paths from $i$ to $j$ with the arrow $\alpha: j \raw l$.
	\item Define the \emph{injective representation} $I(i)$ at vertex $i$  as
\[I(i) = (I(i)_j, I(i)_{\alpha})_{j \in Q_0, \alpha \in Q_1} \]
where $I(i)_j$ is the $\C$-vector space with basis the set of all paths from $j$ to $i$ in $Q$. For $\alpha : j \raw Q$, $I(i)_{\alpha} : I(i)_j \raw I(i)_l$ is the linear map defined on the basis by deleting the arrow $\alpha$ from those paths from $j$ to $i$ which start with $\alpha$ and sending to zero the path that do not start with $\alpha$.
\end{enumerate}
\end{definition}

\begin{example}
For $Q_2$ in Example \ref{ex:2quiver}:
\[
S(1) = 
\begin{tikzcd}
\C \arrow[r, yshift=0.7ex] \arrow[r, yshift=-0.7ex]
& 0
\end{tikzcd}
,
S(2) = 
\begin{tikzcd}
0 \arrow[r, yshift=0.7ex] \arrow[r, yshift=-0.7ex]
& \C
\end{tikzcd}
\]
\[
P(1) = 
\begin{tikzcd}
\C \arrow[r, yshift=0.7ex] \arrow[r, yshift=-0.7ex]
& \C^2
\end{tikzcd}
,
P(2) = 
\begin{tikzcd}
0 \arrow[r, yshift=0.7ex] \arrow[r, yshift=-0.7ex]
& \C
\end{tikzcd}
= S(2)
\]
\[
I(1) = 
\begin{tikzcd}
\C \arrow[r, yshift=0.7ex] \arrow[r, yshift=-0.7ex]
& 0
\end{tikzcd}
=S(1)
,
I(2) = 
\begin{tikzcd}
\C^2 \arrow[r, yshift=0.7ex] \arrow[r, yshift=-0.7ex]
& \C
\end{tikzcd}
\]

\end{example}

There is one handy property about the indecomposable injectives for calculation.

\begin{lemma} \label{thm:proj}
Let $V$ be a representation of $Q$. Then there are natural isomorphisms
\[ \Hom (P(i), V )\cong V_i, \qquad
\Hom (V, I(i) )\cong V_i^*, \]
where $V_i$ is the vector space associated to vertex $i$.
\end{lemma}

Next we will consider the projective resolution of a quiver representation. 

\begin{definition}
Let $V$ be a representation of $Q$.
A \emph{projective resolution} of $V$ is an exact sequence
\[
\cdots \raw P_3 \raw P_2 \raw P_1 \raw P_0 \raw V \raw 0,
\]
A \emph{injective resolution} of $V$ is an exact sequence
\[
0 \raw V \raw I_0 \raw I_1 \raw I_2 \raw I_3 \raw \cdots,
\]
where each $I_i$ is an injective representation.
\end{definition}

In particular, as we are considering acyclic quivers, the corresponding path algebra is hereditary. In particular we have the following theorem:

\begin{theorem} \cite[Theorem 2.15]{schiffler_book}
Let $C$ be a representation of $Q$.
There exists a projective resolution of $D$ of the form
\[ 0 \longrightarrow P_1 \longrightarrow P_0 \longrightarrow C \longrightarrow 0 ,\]
where $P_1 = \bigoplus_{\alpha \in Q_1} (\dim C_{s(\alpha)}) P(t (\alpha))$ and 
$P_0 = \bigoplus_{i \in Q_0} (\dim (C_i)) P(i)$. Similarly, there exists an injective resolution of $D$ of the form
\[
0 \longrightarrow D \longrightarrow I_0 \longrightarrow I_1 \longrightarrow 0, 
\]
where $I_0 = \bigoplus_{i \in Q_0} (\dim (D_i)) I(i)$ and $I_1 = \bigoplus_{\alpha \in Q_1} (\dim D_{t(\alpha)}) I(s (\alpha))$.
\end{theorem}

\begin{remark} \label{rk:inj}
Recall in \eqref{eqn:calE}, we defined $\calE(\textbf{d}) = \chi ( \cdot, \textbf{d})$. Note that $\calE$ is expressing the above injective resolution, which is, 
\begin{equation} \label{eq:gg}
    \calE(\textbf{d})_i = \dim (D_i)- \sum_{\alpha: i \rightarrow j}\dim D_j
\end{equation}
\end{remark}

Let $D$ be any representation of $Q$. Applying $\Hom ( \cdot, D)$ to the above projective resolution and using Lemma \ref{thm:proj}, 
we get the following exact sequence
\[
	0 \raw \Hom (C, D) \raw \Hom (P_0, D) \raw \Hom(P_1, D) \raw \Ext^1 (C, D) \raw 0.
\]
By using Lemma \ref{thm:proj}, we have
\begin{equation}
\chi (C, D) = \dim  \Hom (C, D) - \dim \Ext^1 (C, D) .
\end{equation}

For the definition of Nakayama functor in the next section, let us define \emph{minimal projective resolution}.
\begin{definition}
Let $M \in \rep Q$. A \emph{projective cover} of $M$ is a projective representation $P$ together with a surjective morphism $g: P \raw M$ with the property that, whenever $g': P' \raw M$ is a surjective morphism with $P'$ projective, then there exists a surjective morphism $h: P' \twoheadrightarrow P$ such that the diagram commutes, i.e. $gh = g'$.
\[\begin{tikzcd}
				& P' \arrow[dl, "h"] \arrow[d, "g'"] & \\
 P \arrow[r, "g"] & M  \arrow[r] \arrow[d] &  0   \\
             	& 0 &  
\end{tikzcd}\]

\end{definition}
Then, we have the following
\begin{definition}
A projective resolution 
\[ \cdots \raw P_3 \xrightarrow{f_3} P_2 \xrightarrow{f_2} P_1 \xrightarrow{f_1} P_0 \xrightarrow{f_0} M \rightarrow  0 \]
is called \emph{minimal} if $f_0: P_0 \raw M$ is a projective cover and $f_i: P_i \raw \ker f_{i-1}$ is a projective cover for every $i >0$.
\end{definition}

\subsection{Auslander-Reiten theory} \label{sec:arquiver}

\subsubsection{Auslander-Reiten translation}
Let $Q^{op}$ be the quiver obtained from $Q$ by reversing each arrow. 
Define
\[
	D = \Hom_{\C} ( - , \C) : \rep Q \longrightarrow \rep Q^{op}
\]

Let $P = \bigoplus_{i \in Q_0} P(i)$. 
We have that $\Hom ( V, P)$ is a representation $(H_i, \phi_{\alpha^{op}})$ of $Q^{op}$ for $V$ a representation of $Q$,
where $H_i = \Hom (V, P(i))$ for every $i \in Q_0$ and
 for $\alpha: i \raw j$ in $Q$,
 define a morphism $\alpha: P(j) \raw P(i)$ by $\alpha_{l} : P(j)_l \raw P(i)_l$ by composing $\alpha$ with the paths from $j$ to $l$. 
 Then define 
$\phi_{\alpha^{op}}: H_j \raw H_i$ as $\phi_{\alpha^{op}} (f) = \alpha \circ f$. That is we have the commutative diagram
\[
\begin{tikzcd}
V \arrow[r, "f"] \arrow[dr, "\phi_{\alpha^{op}} (f)"'] & P(j) \arrow[d, "\alpha"] \\
 & P(i)
\end{tikzcd}
\]

Then we define the \emph{Nakayama functor} as
\[
	\nu = D \Hom (-, P) : \rep Q \longrightarrow \rep Q.
\]

\begin{example}
Consider $Q_2$ again. We take 
$P(2) = 
0 \rraws \C
$.
Then $\Hom (P(2), P)$ gives us 
$\C^2 \leftleftarrows \C $.
Taking $D$ will give us 
$ \C^2 \rraws \C $,
which is $I(1)$. Notice that this example also suggests that $\nu$ maps projectives to injectives which is true in general.
\end{example}

Now we are ready to define Auslander-Reiten translation.

\begin{definition}
Let
\[ 0 \xrightarrow{p_1} P_1 \xrightarrow{p_0} P_0 \rightarrow V \raw 0\]
be a minimal projective resolution. Applying the Nakayama functor, we get an exact sequence
\[
	0 \raw \tau V \raw \nu P_1 \xrightarrow{ \nu p_1} \nu P_0 \xrightarrow{\nu p_0} \nu V \raw 0,
\]
where $\tau V = \ker \nu p_1$ is called the Auslander-Reiten translate of $M$ and $\tau$ is the \emph{Auslander-Reiten translation}.
\end{definition}

\begin{example} \label{ex:2quiver_nakayama}
Let us calculate $\tau (\C^2 \rightrightarrows \C^3)$ where $\C^2 \rightrightarrows \C^3$ is the indecomposable representation. First we have the minimal projective resolution
\[ 0 \xrightarrow{p_1}  (0 \rraws \C) \xrightarrow{p_0} (\C \rraws \C^2 )^2
\rightarrow (\C^2 \rightrightarrows \C^3) \raw 0.\]
Applying the Nakayama functor will give us
\[
	0 \raw \tau (\C^2 \rightrightarrows \C^3) \raw (\C^2 \rraws \C)
	 \xrightarrow{ \nu p_1} (\C \rraws 0)^2
	  \xrightarrow{\nu p_0} 0 \raw 0.
\]
Thus $\tau (\C^2 \rightrightarrows \C^3) = 0 \rraws \C$. 
\end{example}

There is a fundamental result called the Asulander-Reiten Formula to relate short exact sequences and morphisms in the module category. This result also give us a way to compute $Ext^1$ which is heavily used in this article. 

\begin{theorem} \cite[Theorem 7.18, Auslander-Reiten formulas]{schiffler_book} \label{thm:ARformula}
Let $V, W$ be $\C Q$-modules. 
We define
\[
	 \underline{\Hom} (V,W ) = \Hom (V,W)/ P(V,W),
	 \]
\[	  \overline{\Hom} (V,W ) = \Hom (V,W)/ I(V,W),
\]
where $P(V, W)$ ($I(V,W)$) is the set of morphisms $f \in \Hom (V,W)$ such that $f$ factors through a projective (injective) $\C Q$-module. 

Then there are isomorphisms
\[ 
	\Ext^1 (V, W) \cong D \underline{\Hom} (\tau^{-1}W, V) \cong D \overline{\Hom} (W, \tau V). 
\]
\end{theorem}

\subsubsection{Almost split sequence}
There is a canonical extension between a module and its Auslander-Reiten translate. 
We have a notion of almost split sequence to describe the short exact sequence.

\begin{definition}
A morphism $f: L \raw E$ is called \emph{left minimal almost split} if
\begin{enumerate}
	\item $f$ is not a section, i.e. there is no morphism $h: E \raw L$ such that $hf = id_{L}$;
	\item for each morphism $u: L \raw U$ in $\modu \C Q$ which is not a section, there exists a morphism $u': E \raw U $ such that $u' f =u$;
	\item if $h : E \raw E$ is such that $hf =f$ then $h$ is an automorphism of $E$.
\end{enumerate}
Similarly, a morphism $g: E \raw F$ is called \emph{right minimal almost split} if
\begin{enumerate}
	\item $g$ is not a retraction, i.e. there is no morphism $h: F \raw E$ such that $gh = id_{F}$;
	\item for each morphism $v: V \raw F$ in $\modu \C Q$ which is not a retraction, there exists a morphism $v': V\raw E $ such that $g v' = v$;
	\item if $h : E \raw E$ is such that $g h = g$ then $h$ is an automorphism of $E$.
\end{enumerate}
\end{definition}

Then we can define almost split sequence

\begin{definition}
A short exact sequence in $\modu \C Q$ 
\[ 0 \longrightarrow L \xrightarrow{f} M \xrightarrow{g} N \longrightarrow 0 \]
is an \emph{almost split sequence} if $f$ is a left minimal almost split morphism and $g$ is a right minimal almost split morphism.
\end{definition}
Let $M =\oplus M_i$, where $M_i$ are indecomposable. Then there exists $M_i$ such that both $L \raw M_i = M/ (\oplus_{j \neq i} M_i)$ and $M_i \xrightarrow{g|_{M_i}} N$ are not zero. If not, we can decompose $M = A \oplus B$, where $A = \im f$ and $B = \ker g$. Then the sequence would be split which contradicts to the definition of almost split exact sequence.

The following theorem gives us the existence of almost split sequences. 
\begin{theorem} \label{thm:almostsplit} \cite[Theorem 7.26]{schiffler_book}
For any finite quiver $Q$ without oriented cycle, there exists a bijection $\tau$ from the indecomposable non-projective representations to indecomposable non-injective representations such that for each non-projective indecomposable $V$, there exists an almost split sequence
\[ 0 \raw \tau V \raw E \raw V  \raw 0 ,\]
where $\tau$ is the Auslander-Reiten translation.
\end{theorem}

\begin{example}
Continuing our calculation in Example \ref{ex:2quiver_nakayama}, we have the almost split sequence
\[
0 \longrightarrow (0 \rraws \C) \longrightarrow (\C \rraws \C^2)^2 \longrightarrow (\C^2 \rraws \C^3) \longrightarrow 0.
\]
\end{example}

\subsubsection{Auslander-Reiten quiver}
Let $V$ and $W$ be indecomposable modules in $\modu \C Q$. 
A homomorphism $f: V \raw W$ is in $\modu \C Q$ is \emph{irreducible} if $f$ is neither a section nor a retraction; and if $f = gh$ for some morphism $h: V \raw Z$, $g: Z \raw W$, then either $h$ is a section or $g$ is a retraction.

Define $\Irr (V,W)$ as the set of irreducible morphisms from $V$ to $W$. It is called the space of \emph{irreducible morphisms}. 
\begin{definition} \cite[Definition IV 4.6]{eltofrep1} \label{def:arquiver}
Consider the path algebra $\C Q$, where $Q$ is a finite acyclic quiver. The quiver $\Gamma (\modu \C Q)$ is defined as:
\begin{itemize}
	\item The points of $\Gamma (\modu \C Q)$ are the isomorphism classes $[V]$ of indecomposable modules $V$ in $\modu \C Q$.
	\item Let $[V]$, $[W]$ be the points in $\Gamma (\modu \C Q)$ corresponding to the indecomposable modules $V$, $W$ in $\modu \C Q$. The arrows $[V] \raw [W]$ are in bijective correspondence with the vectors of a basis of the $K$-vector space $Irr (V, W)$.
\end{itemize}
The quiver $\Gamma (\modu A)$ of the module category $\modu \C Q$ is called the \emph{Auslander-Reiten quiver} of $\C Q$. We will write AR quiver for shorthanded notation. 
\end{definition}

\begin{prop} \cite[Proposition VIII 2.1]{eltofrep1} \label{thm:AR_component}
The Auslander-Reiten quiver $\Gamma (\modu \C Q)$ of $\C Q$ contains a connected component $\calP (\C Q)$ where
\begin{itemize}
	\item for every indecomposable $ \C Q$-module $V$ in $\calP ( \C Q)$, there exist a unique $t \geq 0$ and a unique $a \in Q_0$ such that $V \cong \tau ^{-t} P(a)$.
	\item $\calP ( \C Q)$ contains a subquiver consisting of all the indecomposable projective $ \C Q$-modules; and
	\item $\calP ( \C Q)$ is acyclic.
\end{itemize}
$\Gamma (\modu \C Q)$ also contains a connected component $ \calI (\C Q)$, where
\begin{itemize}
	\item for every indecomposable $ \C Q$-module $W$ in $\calI ( \C Q)$, there exist a unique $s \geq 0$ and a unique $a \in Q_0$ such that $W \cong \tau^s I(a)$.
	\item $\calI ( \C Q)$ contains a subquiver consisting of all the indecomposable injective $ \C Q$-modules; and
	\item $\calI ( \C Q)$ is acyclic.
\end{itemize}
Furthermore, $\calP ( \C Q)= \calI ( \C Q)$ if and only if $Q$ is of Dynkin type.
\end{prop}

$\calP (\C Q)$ is called the \emph{pre-projective} component of $\Gamma (\modu \C Q)$ and
$\calI (\C Q)$ is called the \emph{pre-injective} component of $\Gamma (\modu \C Q)$.
An indecomposable $\C Q$-module is called \emph{pre-projective} if it belongs to $\calP (\C Q)$ and 
it is called \emph{pre-injective} if it belongs to $\calI (\C Q)$.
From the above, if $Q$ is of Dynkin type, then $\calP ( \C Q)= \calI ( \C Q)$ from the theorem above.  
Then $\Gamma (\modu \C Q)$ is connected and $\Gamma (\modu \C Q) =\calP ( \C Q)= \calI ( \C Q)$ from the properties of projectives and injectives.
If $Q$ is not of Dykin type, there are representations which is neither pre-projective or pre-injective.
We will call the connected component $\calR (\C Q)$ of $\Gamma (\modu \C Q)$ to be \emph{regular}
if $\calR (\C Q)$ contains neither projective nor injective modules. An indecomposable representation is called \emph{regular} if it belongs to a regular component of $\Gamma (\modu \C Q)$ and an arbitrary indecomposable representation is called \emph{regular} if it is a direct sum of indecomposable regular representations.

\begin{example}
For Kronecker 2-quiver, we have the Auslander-Reiten quiver as follows
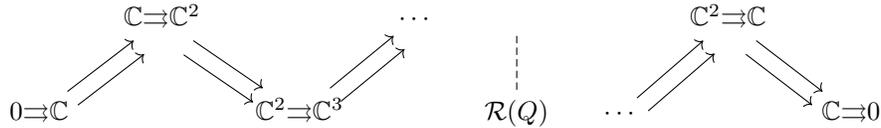
\[
\begin{tikzcd}[column sep=small]
   &  \C \rraws \C^2 \arrow[dr, shift right=1.5ex] \arrow[dr]  &                  & \cdots  & \arrow[d, dashed, no head]  & & \C^2 \rraws \C \arrow[dr, shift right=1.5ex] \arrow[dr] & \\
 0 \rraws \C \arrow[ur, shift right=1.5ex] \arrow[ur] &        & \C^2 \rraws \C^3 \arrow[ur, shift right=1.5ex] \arrow[ur] & & \calR (Q) & \cdots \arrow[ur, shift right=1.5ex] \arrow[ur] & & \C \rraws 0
\end{tikzcd}
\]
Note that the left component contains $P(1)$ and $P(2)$. Thus that is the pre-projective component. The right component contains $I(1)$ and $I(2)$ which means it is the pre-injective component. The middle component is the regular component which contains all the $\C^k \rraws \C^k$, where $k \geq 0$.
\end{example}

\subsubsection{Computing \texorpdfstring{$\Hom$}{Hom} and \texorpdfstring{$\Ext^1$}{Ext} group}  \label{rmk:ext_zero}

We are going to state the properties of the Auslander-Reiten quiver which is useful for computing the $\Hom$ and $\Ext^1$ group.

Let $V$ and $W$ be two indecomposable $\C Q$-module. A path in $\modu A$ from $V$ to $W$ is a sequence
\[ V = V_0 \xrightarrow{f_1} V_1 \xrightarrow{f_2} \cdots \xrightarrow{f_t} V_t =W \] 
where all the $V_i$ are indecomposable, and all the $f_i$ are nonzero nonisomorphisms. In this case, $V$ is called a \emph{predecessor} of $W$ and $W$ is called a \emph{successor} of $V$ in $\modu A$.

\begin{prop} \cite[VIII Lemma 2.5]{eltofrep1} \label{thm:ARprop}
\begin{itemize}
	\item Let $\calP$ be a preprojective component of the quiver $\Gamma (\modu \C Q)$ and $V$ be an indecomposable module in $\calP$. Then the number of predecessors of $V$ in $\calP$ is finite and any indecomposable $\C Q$-module $L$ such that $\Hom_{\C Q} (L, V) \neq 0$ is a predecessor of $V$ in $\calP$. In particular, $\Hom _{\C Q}(L, V) =0$ for all but finitely many nonisomorphic indecomposable $\C Q$-modules $L$.
	\item Let $\calI$ be a preinjective component of the quiver $\Gamma (\modu \C Q)$ and $N$ be an indecomposable module in $\calI$. Then the number of successors of $W$ in $\calI$ is finite and any indecomposable $\C Q$-module $L$ such that $\Hom_{\C Q} (W,L) \neq 0$ is a predecessor of $W$ in $\calI$. In particular, $\Hom _{\C Q}(W,L) =0$ for all but finitely many nonisomorphic indecomposable $\C Q$-modules $L$. 
\end{itemize}
\end{prop}

From the theorem above, we can tell immediately that if $V$ is a predecessor of $W$, then $\Hom(W,V) =0 $, where $V$ and $W \in \calP$ or $V$ and $W \in \calI$. We can also understand the $\Hom $ group between different connected component in $\Gamma (\modu \C Q)$.

\begin{prop} \cite[Corollary VIII 2.13]{eltofrep1} \label{thm:PRI}
Let $P$, $I$ and $R$ be three indecomposable $\C Q$-module. 
\begin{enumerate}
	\item If $P$ is preprojective and $R$ is regular, then $\Hom (R, P) = 0$.
	\item If $P$ is preprojective and $I$ is preinjective, then $\Hom (I, P) = 0$.
	\item If $R$ is regular and $I$ is preinjective, then $\Hom (I, R) = 0$.
\end{enumerate}
We may write as $\Hom (\calR, \calP) = \Hom (\calI, \calP) = \Hom ( \calI , \calR) = 0$. 
\end{prop}
Let us abuse the term `predecessor' and `sucessor' to say elements in $\calP$ are predecessors of elements in $\calR$, $\calI$ and elements in $\calR$ are predecessors of elements of $\calI$. We will do the same for sucessor.

From the properties of the AR quivers, we observe the following properties which will be useful for computations in later chapters.

\begin{lemma} \label{thm:lemmaPRI}
\[\Ext^1(\calP, \calR) = \Ext^1( \calP, \calI) = \Ext^1 (\calR, \calI) = 0.\]
\end{lemma}

\begin{proof}
Let us first show $\Ext^1(\calP, \calR) =0$. If not, there exists a non-split exact sequence
\[ 0 \longrightarrow R \xrightarrow{f} V \xrightarrow{g} P \longrightarrow 0 \]
where $R \in \calR$, $P \in \calP$ and for some $V \in \rep (Q)$. Decompose $V = \bigoplus V_i$ where $V_i$ are indecomposable. Consider $V_i$ such that $V_i \cap \im f \neq 0$ and $g(V_i) \neq 0$. If such a $V_i$ does not exist, the exact sequence will split. Then if $V_i \in \calR$, $g|_{V_i} : V_i \raw P$ nonzero will contradict Theorem \ref{thm:PRI}.
 Similarly, if $V_i \in \calP$ or $\calI$, it will contradict Theorem \ref{thm:PRI}. Thus $\Ext^1(\calP, \calR) =0$. Repeating the same argument will give us $ \Ext^1( \calP, \calI) = \Ext^1 (\calR, \calI) = 0$.
\end{proof}

\begin{lemma} \label{thm:lemmaext}
If $V, W \in \calP$ or $V, W \in \calI$ and if $\Hom (V, W) \neq 0$, then 
\[\Ext^1(V, W) =0.\]
\end{lemma}

\begin{proof}
Consider $V, W \in \calP$. If $V$ is projective, then  $\Ext (V, W) =0$ for all $2$. Now assume $V$ is not projective. Then by Auslander-Reiten Theorem \ref{thm:ARformula} which tells $\Hom ( W, \tau  V) \neq 0$, i.e. we have a map
\[  W \longrightarrow \tau V  \]
By Theorem \ref{thm:almostsplit} and the remark above, there exists $E $ for $V$ non-projective such that there is an almost split exact sequence
\[0 \longrightarrow \tau V \xrightarrow{f} E \xrightarrow{g} V \longrightarrow 0. \]
We can again decompose $E = \bigoplus E_i$ where $E_i$ are indecomposable. We can further assume $E_i  \in \calP$ for all $i$ by Theorem \ref{thm:PRI}. Pick $E_i$ such that $E_i \cap \im (f) \neq 0 $ and
$g(E_i) \neq 0$. Such an $E_i$ exists because the sequence is non-split. Then $\Hom (\tau V , E_i) \neq 0$ which implies $\tau V$ is a predecessor of $E_i$. Also $\Hom (E_i, V ) \neq 0$ implies $E_i$ is a predecessor of V. Then there exists a path in AR quiver such that
\[ W \raw \tau V \raw \cdots \raw E_i \raw \cdots \raw V \raw \cdots \raw W.\]
The last arrow follows from the assumption $\Hom (V, W) \neq 0$. Then we obtain a cyclic in $\calP$ which contradicts to Theorem \ref{thm:ARprop}. We can repeat a similar argument for $V, W \in \calI$.
\end{proof}

\begin{lemma} \label{thm:extnonzero}
Now if $V, W \in \calP$ or $V, W \in \calI$, assume $\Ext^1 (W,V) \neq 0$, then $V$ is a predecessor of $W$ in the AR quiver. That is $\Hom (W, V) =0$. 
\end{lemma}

\begin{proof}
Consider $V, W \in \calP$. As $\Ext^1 (W,V) \neq 0$, we have a non-split exact sequence
\[ 0 \longrightarrow V \longrightarrow E \longrightarrow W \longrightarrow 0 ,\]
for some $E \in \rep (Q)$. By repeating the arguments in the proofs of the two theorem above, we can decompose $E$ as sums of indecomposable $E_i$ with $E_i \in \calP$. Furthermore, there exists $E_i$ such that $V$ is a predecessor of $E_i$ and $E_i$ is a predecessor of $W$ in $\calP$, then we have a path in $\calP$: 
\[ V \raw \dots \raw E_i \raw \cdots \raw W.\]
This shows that $V$ is a predecessor of $W$. This holds similarly for $V, W \in \calI$.
\end{proof}

Combining all the lemmas above, we have if $V, W \in \calP$ or $V, W \in \calI$, if $\Hom (V, W) \neq 0$, then $\Ext^1(V, W) =0$. However if $\Ext^1 (W,V) \neq 0$, then $\Hom (W, V) =0$. This is saying that
\[ \chi(V, W) = \dim \Hom (V, W) -\dim \Ext^1 (V, W)\]
actually equals to either $\dim \Hom (V, W)$ or $-\dim \Ext^1 (V, W)$.

\section{Introduction to scattering diagrams and theta functions}
\subsection{Scattering Diagram} \label{sec:scatteringdiag}

In this section, we will go over the definition of scattering diagrams which will be associated to cluster algebras $\calA( \epsilon)$.
As this article focus on the cluster algebras associated to quivers, we would assume the exchange matrix $\epsilon$ to be skew-symmetric. 

Let $N$ be a rank $n$ lattice, $M = \Hom (N, \Z)$. 
Denote $M_{\R} = M \otimes \R$, $N_{\R} = N\otimes \R$. 
Take $\Bbbk$ to be a field of characteristic 0.
Now a \emph{seed data} $\seed = (e_i)$ would be a basis $e_1, \dots , e_n $ of $N$. 
Denote
\[
N^+ = \{ \sum a_i e_i | \quad a_i \geq 0, \sum a_i >0 \}
\]
Let $f_1, \dots, f_n$ be the dual basis in $M$.
Given $\textbf{m}\in M$, we would denote $z^m \in \Bbbk[M]$ as $A_1^{m_1}\cdots A_\n^{m_n}$ if $m=m_1f_1+\dots +m_nf_n$.

We further fix a skew-symmetric bilinear form on $N$
\[ \{ \cdot, \cdot \} : N \times N \raw \Z .\]
Define the skew symmetric matrix $\epsilon$ with $\epsilon_{ij} = \{ e_i, e_j \}$. 
This is the same matrix as in \eqref{eq:quivermat}. 
The skew-symmetric form induces a map
\begin{align*}
		p^* :  N \longrightarrow  M , \qquad
		       \textbf{n} \mapsto  \{ \textbf{n} , \cdot \}
\end{align*}

To have scattering diagrams well-defined, $p^*$ is needed to be injective which does not hold in general.
We will first state the definition of scattering diagrams with assuming the injectivity of $p^*$.
In the next section (Section \ref{sec:ghkk}), we would switch the seed data in $N$ to the `doubling lattice' $\wN$.
Then the injectivity of $p^*$ would naturally follows and then the scattering diagrams defined below will be well-defined without extra assumption. 

\begin{definition}
A wall in $M_{\R} = M \otimes_{Z} \R$ is a pair $(\frakd, f_{\frakd})$ where
\begin{itemize}
	\item $\frakd \subseteq M_{\R}$, support of the wall, is a convex rational polyhedral cone of codimension one, contained in $n^{\perp}$ for some $\textbf{n} \in N^+$, 
	\item $f_{\frakd} \in \C [[z^{p^*(e_i)}, i = 1, \dots, n]]$ such that $ f_{\frakd} = 1+ \sum_{k \geq 1} c_k z^{k p^*(\textbf{n})}$ for some $c_k \in (A_1, \cdots A_\n)$. 
\end{itemize}
A wall $(\frakd, f_{\frakd})$ is called \emph{incoming} if $p^* (n) \in \frakd$. Otherwise it is called \emph{outgoing}.
\end{definition}

\begin{definition}  \label{def:scattering_diagram}
  A scattering diagram $\frakD$ is a collection of walls such that, for each $k \geq  0$, the set
  \[
    \{ (\frakd, f_{\frakd}) \in \frakD\, |\, f_{\frakd} \neq 1 \bmod (z^{p^* (e_i)})^k \}
  \] 
  is finite. The support of a scattering diagram, Supp$(\frakD)$, is the union of the supports of its walls. We will write
  \[
  	Sing (\frakD) = \bigcup_{\frakd \in \frakD} \partial \frakd \cup \bigcup_{\frakd_1, \frakd_2 \in \frakD ,  \dim \frakd_1 \cap \frakd_2 = n-2} \frakd_1 \cap \frakd_2.
  \]
\end{definition}

\subsubsection*{Path-ordered product/ Wall crossing}
Next we consider a smooth immersion
\[  \gamma: [0,1] \rightarrow M_{\mathbb{R}} \backslash  \{0 \}  \]
with endpoints not in the support of $\frakD$. Assume $\gamma$ is transversal to each wall of $\frakD$. For each power $k \geq 1$, we can find the sequence of numbers
$ 0< t_1 \leq t_2 \leq \cdots \leq t_s < 1 $ with $\gamma(t_i)\in\frakd_i$
for some $i$ with
$f_{\frakd_i} \neq 1 \text{ mod } \frakm^k   $ and $\frakd_i\not=\frakd_j$ whenever
$t_i=t_j$. For each $i$, define $\frakp_{\frakd_i}\in {\mathrm{Aut}}_{\Bbbk-alg}\big(\Bbbk [[z^{p^* (e_i)}]]\big)$, the \emph{path-ordered product} as
\begin{equation} \label{eqn:pathordered}
\frakp_{\gamma, \frakd_i} (z^m) = z^m f_{\frakd_i}^{\langle m, n_0 \rangle },
\end{equation}
where the lattice point $n_0 \in N$ is primitive, annihilates the tangent space to $\frakd_i$, and is uniquely determined by the sign convention $ \langle n_0, \gamma'(t_i) \rangle <0$. 
Take $\frakp_{\gamma, \frakD,k} := \frakp_{\frakd_s} \circ  \cdots \circ \frakp_{\frakd_1}$. We can then define the \textit{path-ordered product} as
\begin{equation} \label{eqn:path_ordered}
	\frakp_{\gamma, \frakD} = \lim_{k \rightarrow \infty} \frakp _{\gamma, \frakD,k}.
\end{equation}

With the path-ordered product, we can talk about equivalence among scattering diagrams.
\begin{definition}
Two scattering diagram $\frakD$, $\frakD'$ are \emph{equivalent} if $\frakp_{\gamma, \frakD} = \frakp_{\gamma, \frakD'}$ for all path $\gamma$ for which both are defined.
\end{definition}

Note that the path-ordered product depends on the path $\gamma$ on $\frakD$. We will call $\frakD$ consistent if the automorphism only depends on the endpoint of the path $\gamma$.

\begin{definition} \label{def:consistent}
A scattering diagram is \emph{consistent} if $\frakp _{\gamma, \frakD}$ only depends on the endpoints of $\gamma$ for any path $\gamma$ for which $\frakp _{\gamma, \frakD}$ is well defined.
\end{definition}

Not every scattering diagram $\frakD$ is consistent; however, there always exists a consistent scattering diagram which contains $\frakD$ as seen in the following theorem.

\begin{theorem}(Kontsevich-Soibelman \cite{KS}, Gross-Siebert \cite{GS}) \label{th:KS}
Given a scattering diagram $\frakD$, there always exists a consistent scattering diagram $\frakD'$ which contains $\frakD$ such that $\frakD'\setminus\frakD$ only consists only of outgoing walls, and it is unique up to equivalence.
\end{theorem}

\subsubsection{Cluster algebra scattering diagrams} \label{sec:ghkk}

As stated in above, we will consider the \emph{double} of the lattice $N$ by considering
\begin{align*}
    \wN = N \oplus M, 
\end{align*}
and the corresponding skew-symmetric bilinear form as 
\begin{align*}
     \{ (n_1, m_1) , (n_2, m_2) \}  = \{n_1, n_2 \} + \langle n_1, m_2 \rangle - \langle n_2 , m_1 \rangle.  
\end{align*}
Then the map
\begin{align*}
		\wpp^* :  N\oplus M \longrightarrow  M \oplus N, \qquad
		       (n, m) \mapsto  \{ (n, m), \cdot \}
\end{align*}
would then be injective.

With an initial seed data $\seed = (e_i)$ in $N$, the corresponding seed in $\wN$ is
\[
\wseed = \big( (e_1, 0), \dots, (e_\n,0), (0,f_1), \dots, (0,f_\n)   \big),
\]
where $(f_i)$ is the dual basis of $(e_i)$. 
Note that $\wseed $ gives a basis of $\wN$.
The dual space $\widetilde{M}_\R \cong M_\R \oplus N_\R$.
We will denote $z^{(m,n)} = A_1^{m_1} \cdots A_\n^{m_\n} X_1^{n_1} \cdots X_\n^{n_\n}$ for $\textbf{m} = m_1f_1 + \dots +m_\n f_\n $ and $\textbf{n} = n_1 e_1 + \dots n_\n e_\n$.

We will now construct a scattering diagram. Set
\begin{equation} \label{eqn:construction}
	\frakD_{in, \seed}^{\calA_{\text{prin}}}= \{ ( (e_i,0)^{\perp}, 1+z^{\wpp^*(e_i,0)} ) | \quad i =1, \dots \n \}.
\end{equation}

Then by Theorem \ref{th:KS}, there exists a consistent scattering diagram $\Dprin$ containing $\frakD_{in, \seed}^{\calA_{\text{prin}}} $. 
We can describe it more explicitly.

\begin{prop} \cite{ghkk}
There exists a consistent scattering diagram $\Dprin$ containing $\frakD_{in, \seed}^{\calA_{\text{prin}}}  $, and $\Dprin \setminus \frakD_{in, \seed}^{\calA_{\text{prin}}}  $ consists of only outgoing walls. 
$\frakD_{in, \seed}^{\calA_{\text{prin}}} $ is unique up to equivalent and it is equivalent to a scattering diagram all of whose walls $(\frakd, f_{\frakd})$ are of the form $\frakd \subseteq (n,0)^{\perp}$ for some $n \in N^+$, and $f_{\frakd}= (1+z^{\wpp^*(n,0)})^c$, for some $c $ a positive integer. In particular, all nonzero coefficients of $f_{\frakd}$ are positive integers.
\end{prop}
Note that in particular the function $f_{\frakd}$ attached to a wall $\frakd$ in $\frakD_{ \wseed}$ is a power series in $z^{(p_1^*(n), n)}$ for some $n \in N^+$.

In particular, given a seed data $s=(e_1, \dots, e_n)$, define
\begin{equation} \label{eq:positivechamber}
    \calC_s^+ := \calC^+ = \{ m \in M_{\R} | \langle e_i, m \geq 0 \rangle, i=1, \dots n\}
\end{equation}
We will call $\calC_s^+$ as the positive chamber. 


\subsubsection*{Example: Two-dimensional case} \label{sec:2Ddiag}
Let us illustrate what we have just defined in two dimensions explicitly. 
Consider $N \cong \Z^2$, i.e. $\wN \cong \Z^4$, $\wM \cong \Z^4$.
Consider the cluster algebra $\calA (b)$ defined in Example \ref{ex:2alg}. 
As we are going to work with skew-symmetric case, we have $b=c$. 
The exchange matrix $B$ in the initial seed of $\calA (b)$ would be 
$B=
\left(
\begin{array}{c c}
0 & b \\ -b & 0
\end{array}
\right)$.
The corresponding skew-symmetric form for $\wN$ would be
\[
\left(
\begin{array}{c c c c}
0 & b   & -1 &0\\ -b & 0 & 0 & -1 \\
1 & 0 & 0 & 0 \\ 0 & 1 & 0 & 0
\end{array}
\right).\]

Take $\seed = \big( (1,0), (0,1) \big) $, and then  $\wseed= \big( (1,0, 0, 0), (0,1, 0, 0), (0,0, 1, 0), (0,0, 0, 1)  \big) $. 
The associated scattering diagrams would lie in the 4-dimensional vector space $\wM_{\R} \cong M_{\R} \oplus N_{\R}$.
First, we have
\[
	\frakD_{in, \seed}^{\calA_{\text{prin}}} = \{ \big( (1,0,0,0)^{\perp} , 1+z^{(0,b,1,0)} \big), \big( (0,1,0,0)^{\perp} , 1+z^{(-b, 0, 0,1)} \big)\}.
\]
Then we would obtain the consistent scattering diagrams $\Dprin$ by Theorem \ref{th:KS}. 
To visualise the 4-dimensional picture, let us consider the projection $M_{\R} \oplus N_{\R} \rightarrow M_{\R}$, $(\textbf{m},\textbf{n}) \mapsto \textbf{m}$. By considering $b=1$, we obtain the standard scattering diagram for the $A_2 $ quiver: 
\begin{figure}[H]
  \centering
  \begin{tikzpicture}
    \draw (2,0) -- (-2,0) node[left] {$1+A_1^{-1}X_2$};
    \draw (0,-2) -- (0,2) node[above] {$1+A_2X_1$};
    \draw (0,0) -- (2,-2) node[below right] {$1+A_1^{-1}A_2X_1X_2$};
  \end{tikzpicture}
  \caption{The scattering diagram for $\calA (1)$.} 
  \label{fig:diagex}
\end{figure}
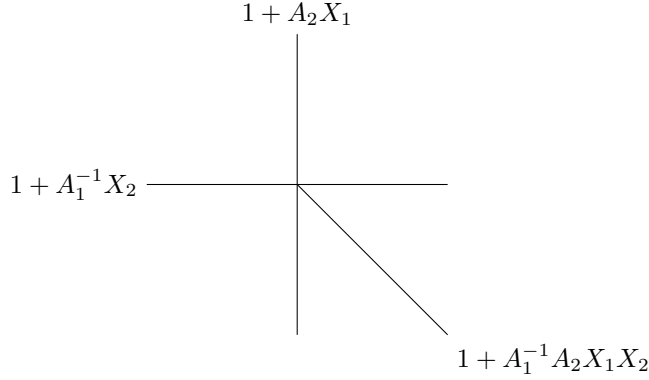

While this example portrays a scattering diagram with finitely many rays, the diagram for $\calA (b)$ will consist of an infinite number of rays precisely when $b^2\ge 4$. 
Figure \ref{fig:22case} illustrate the diagram for $\calA (2)$ which is with infinitely rays of slope $(k, -(k+1))$, $(k+1, -k)$ and $(1,1)$. In general if $b^2 \geq 4$ then there are infinitely many discrete outgoing rays converge to the ray of slopes $-(b \pm \sqrt{b^2-4})/ 2$. 
Within these two rational slopes, the set of rays of rational slope appear is dense.
This region is called the \emph{badlands}.

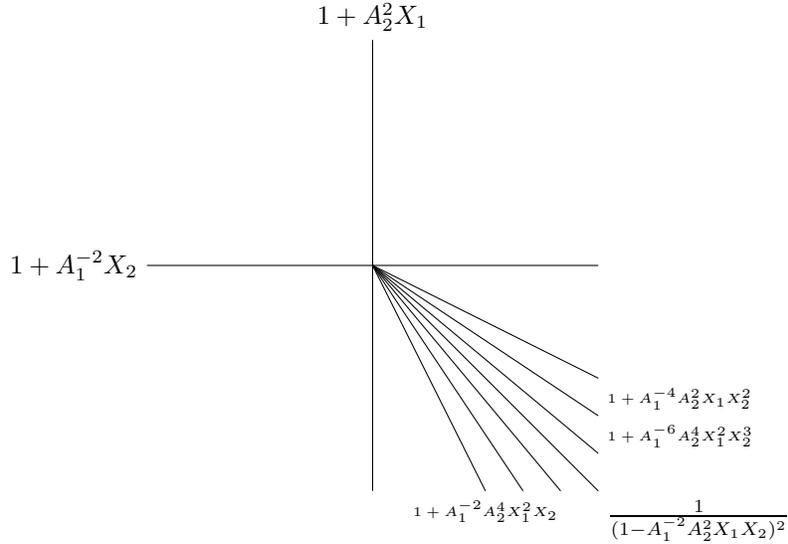
\begin{figure}[H]
  \centering
  \begin{tikzpicture}
    \draw (3,0) -- (-3,0) node[left] {$1+A_1^{-2}X_2$};
    \draw (0,-3) -- (0,3) node[above] {$1+A_2^2X_1$};
    \draw (0,0) -- (3,-3) node[below right] {$\frac{1}{(1-A_1^{-2}A_2^2X_1X_2)^2}$};
    \draw (0,0) -- (3,-1.5) node[below right] {\tiny $1+A_1^{-4}A_2^2X_1X_2^2$};
    \draw (0,0) -- (1.5,-3) node[below] {\tiny $1+A_1^{-2}A_2^4X_1^2X_2$};
    \draw (0,0) -- (3,-2) node[below right] {\tiny $1+A_1^{-6}A_2^4X_1^2X_2^3$};
    \draw (0,0) -- (2,-3) ;
    \draw (0,0) -- (3,-2.5)  ;
    \draw (0,0) -- (2.5,-3)  ;
  \end{tikzpicture}
  \caption{The scattering diagram for $\calA (2)$.} 
  \label{fig:22case}
\end{figure}

\subsection{Broken lines and theta functions} \label{sec:broken}

Broken lines were introduced in \cite{grossp2} as a way of describing holomorphic disks which appear in
mirror symmetry in a tropical manner. Their theory was further developed
in \cite{CPS}, and then used in \cite{GHKlog} and \cite{ghkk} to construct
canonical bases in various circumstances. 

\begin{definition} \label{brokendef}
Let $\frakD$ be a scattering diagram, $m \in \wM \setminus \{0\}$ and $\enpt \in \wM_{\R} \setminus \text{Supp}(\frakD)$. A \emph{broken line} for $m$ with endpoint $\enpt$ is a piecewise linear continuous proper path $\gamma : ( - \infty , 0 ) \rightarrow \wM_{\mathbb{R}} \setminus Sing (\frakD)$ with a finite number of domains of linearity and a collection of monomials. A monomial $c_L z^{m_L} \in \Bbbk[\wM]$ is attached to each domain of linearity $L \subseteq ( - \infty, 0)$ of $\gamma$. The path $\gamma$ and the monomial $c_L z^{\textbf{m}_L}$ need to satisfy the following conditions:
\begin{itemize}
    \item $\gamma(0) = \enpt$.
    \item If $L$ is the first (i.e., unbounded) domain of linearity of $\gamma$, then $c_L z^{\textbf{m}_L} = z^{\textbf{m}}$.
    \item For $t$ in a domain of linearity, $\gamma'(t) = -\textbf{m}_L$.
    \item $\gamma$ bends only when it crosses a wall. If $\gamma$ bends from the domain of linearity $L$ to $ L'$ when crossing $(\frakd, f_{\frakd})$, then $c_{L'}z^{\textbf{m}_{L'}}$ is a term in $\frakp_{\gamma, \frakd} (c_L z^{\textbf{m}_L})  $.
\end{itemize}
For a broken line $\gamma$ with initial slope $m$ and endpoint $\enpt$, denote $I(\gamma) = m$ and $b(\gamma) = \enpt$. Define
$\text{Mono} (\gamma) = c(\gamma)z^{F(\gamma)}$ to be the monomial $c_L z^{m_L}$ attached to the last domain of linearity $L$ of $\gamma$. 
We further define the theta functions as
\begin{equation} \label{eqn:thetaf}
 \vartheta_{\enpt, m} = \sum_{\gamma} \text{Mono} (\gamma),  
\end{equation}
where the sum is over all broken lines for $m$ with endpoint $\enpt$.
\end{definition}

The following summarizes the properties of the theta functions as shown in \cite{CPS} and \cite{ghkk}. 
\begin{prop} \label{thm:propoftheta}
      If $\fD$ is any consistent scattering diagram, $\enpt$ and $\enpt'$ are two general
      irrational points on $M_{\R} \smallsetminus$ Supp$(\fD)$, and $\gamma$ is a path joining $\enpt$ to $\enpt'$, then $\fp_{\gamma, \fD }(\vartheta_{\enpt,m}) = \vartheta_{\enpt', m}$. 
\end{prop}

\begin{prop} \label{thm:zm}
Now consider $\fD^{\cAp}$ a cluster scattering diagram.
If $\enpt$ and $m$ lie in the interior of the same chamber in the cluster complex of $\fD^{\cAp}$, then $\vartheta_{\enpt,m}=z^{m}$.
\end{prop}

The following proposition from \cite{ghkk} leads us the relation between theta functions and cluster monomials.

\begin{prop} \label{rmk:thetaiscluster}
If $\enpt$ lies inside the positive chamber, and $\textbf{m}\in \wM$ lies in one of the chambers which is reachable from the positive chamber, i.e. $\textbf{m} \in \calC^+ \in \Delta_s$ defined in \eqref{eq:positivechamber}. 
Then $\vartheta_{\enpt,m}=\frakp_{\gamma,\frakD}(z^{\textbf{m}})$ for the path $\gamma$ joining from chamber containing $m$ to $\enpt$. 
In this case, $\vartheta_{\enpt,\textbf{m}}$ is a cluster monomial \cite{ghkk}. 
Further if $\enpt$ is in the positive chamber, and $\vartheta_{\enpt, m}$ is a finite sum, then $\vartheta_{\enpt, m}$ is an element of the cluster algebra.

In particular, by combining Proposition \ref{thm:propoftheta} and Proposition \ref{thm:zm}, consider $\gamma$ the path which joins $\textbf{m}$ to a point $\enpt$ in the positive chamber $\calC_\textbf{s}^+$ by passing through finitely many chambers. Then we have
\begin{equation}
    \vartheta_{\enpt, \textbf{m}} = \frakp_{\gamma, \frakD} (z^{\textbf{m}}).
\end{equation}
\end{prop}

\begin{remark} \label{rk:cg}
We are going to see $g$ vectors and $c$ vectors are the generators and normals of the chambers of the scattering diagrams. Even though it is indicated in \cite{ghkk} already, we are giving a proof by employing idea from the cluster algebras world in \cite{cluster4} and \cite{nakanishi2012tropical}.

Fix $\textbf{m}_0 \in M$, and consider the theta function $\vartheta_{(\textbf{m}_0,\textbf{n})}$, where $\textbf{n} \in N$. 
One can see that the coefficients in the $\vartheta_{\textbf{m}_0,\textbf{n}}$ would remain the same for any $\textbf{n} \in N$. This means that the $F$ polynomials defined in \eqref{eq:F} only depend on $\textbf{m}_0$.  
Then we can consider the projection 
\begin{equation}
    \rho: \wM = M \oplus N \rightarrow M.
\end{equation} 
This gives us the $\calA$-theta functions which are the cluster monomials for cluster algebras without principal coefficients.
In particular, we consider $\vartheta_{\enpt, (\textbf{m}_0, 0)}$. It has the form
\[
\vartheta_{\enpt, (\textbf{m}_0, 0)} 
= z^{(\textbf{m}_0, 0)} F\big( z^{(\textbf{m}, \textbf{n})} \big),
\]
where $F$ is a Laurent polynomial with monomials of the form $z^{\wpp^*(n,0)} =  z^{(p^*(n),n)} $, $n \in N^+$.
This is exactly the expression as in \eqref{eq:gvector} if we set $X_i =1$ for all $i$.
Thus we can deduce $\textbf{m}_0$ is the $g$-vector for the cluster monomial $\vartheta_{\enpt, (\textbf{m}_0, 0)}$.
This is why scattering diagram are known as $g$-vector fan. 
By Theorem \ref{thm:tropduality} in \cite{nakanishi2012tropical}, we learnt that the $c$ and $g$ vectors are dual bases. 
So the $c$ vectors are the normals of the walls of the chambers. 
\end{remark}

We will show some computation of some broken lines and theta functions in the two dimensional scattering diagram in Section \ref{sec:2Ddiag}.

\begin{example} \label{ex:11case}
  \label{brokenex}
  Consider the scattering diagram $\fD_{(2)}$ in Figure \ref{fig:22case}
  and let $\enpt$ be a small irrational perturbation of the point $(1, -1.5, 0 , 0)$. 
  There are three broken lines with initial exponent $m = (1,-1, 0,0)$ and endpoint $\enpt$ as shown in
  Figure~\ref{figbrokenex}.
  First of all, we can have a broken line $\gamma_1$ which does not bend.
  Therefore
  \[
    \Mono(\gamma_1) = z^{ (1,-1, 0,0)} = A_1 A_2^{-1}.
  \]
  There is the broken line $\gamma_2$  which bends only at the $x$-axis. Since
  \[ 
    \frakp_{\gamma_2, (0,1,0,0)^{\perp}} (z^{ (1,-1, 0,0)}) = 
    z^{ (1,-1, 0,0)} \big(1+z^{ (-2,0, 0,1)} \big) =  A_1 A_2^{-1} + A_1^{-1} A_2^{-1}X_2,
  \]
  to bend we need to choose the second term and obtain 
  \[
    \Mono(\gamma_2) = z^{ (-1,-1,0,1)}= A_1^{-1} A_2^{-1}X_2.
  \]
  The last broken line $\gamma_3$ bends both at the $x$- and $y$-axes, the
  latter bend coming from
  \[ 
    \frakp_{ \gamma_3, (1,0,0,0)^{\perp}}  (z^{ (-1,-1,0,1)}) =  
    z^{ (-1,-1,0,1)} + z^{ (-1,1,1,1)}.  
  \]
  This time we have 
  \[
    \Mono (\gamma_3) = z^{ (-1,1,1,1)}= A_1^{-1} A_2X_1X_2.
  \]
  Thus the theta function associated to $m = (1,-1,0,0)$ with endpoint point $\enpt$ is 
  \[ 
    \vartheta_{\enpt, (1,-1,0,0)} =  
    A_1 A_2^{-1} +A_1^{-1} A_2^{-1}X_2  +A_1^{-1} A_2X_1X_2 .  
  \]
  Note that we can restrict the theta function to the 
  
  By considering $X_i =1$ for all $i$, we get $\vartheta^{\calA}_{\enpt, (1,-1)} := A_1 A_2^{-1} +A_1^{-1} A_2^{-1}  +A_1^{-1} A_2$.
  This is a projection to the scattering diagram for $\calA$ instead of $\cAp$.
  One can see that this is the cluster variable for the cluster algebra $\calA$ of type $A_2$.
\end{example}

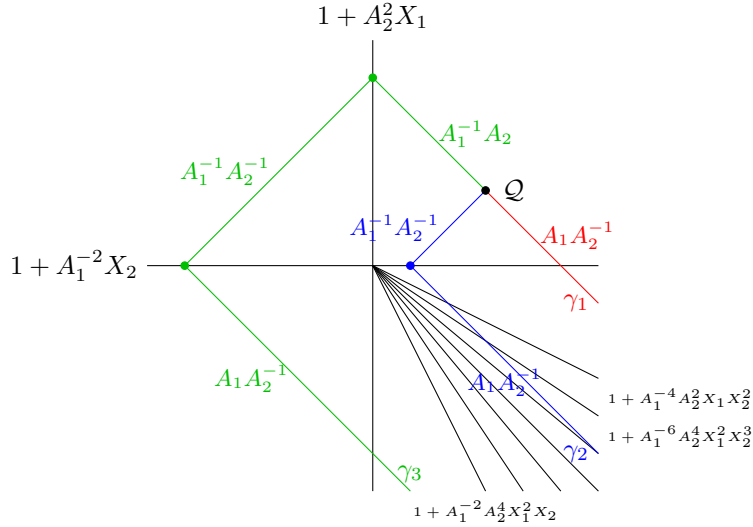
\begin{figure}
  \centering
  \begin{tikzpicture}
    \draw (3,0) -- (-3,0) node[left] {$1+A_1^{-2}X_2$};
    \draw (0,-3) -- (0,3) node[above] {$1+A_2^2X_1$};
    \draw (0,0) -- (3,-3) ;
    \draw (0,0) -- (3,-1.5) node[below right] {\tiny $1+A_1^{-4}A_2^2X_1X_2^2$};
    \draw (0,0) -- (1.5,-3) node[below] {\tiny $1+A_1^{-2}A_2^4X_1^2X_2$};
    \draw (0,0) -- (3,-2) node[below right] {\tiny $1+A_1^{-6}A_2^4X_1^2X_2^3$};
    \draw (0,0) -- (2,-3) ;
    \draw (0,0) -- (3,-2.5)  ;
    \draw (0,0) -- (2.5,-3)  ;
    \draw (1.5,1) node[circle, right] {$\enpt$};
    \draw[red] (3,-0.5) -- (1.5,1);
    \draw[red] (2.1,0.1) node[above right]{\small $A_1 A_2^{-1}$};
    \draw[red] (3,-0.5) node[left]{$\gamma_1$};
    \draw[blue] (0.5,0) -- node[left] {\small $A_1^{-1} A_2^{-1}$}  (1.5,1);
    \draw[blue] (0.5,0) -- node[below] {\small $A_1 A_2^{-1}$} (3,-2.5);
    \draw[blue] (3,-2.5) node[left]{$\gamma_2$};
    \draw[green!75!black] (0,2.5)  -- node[right] {\small $A_1^{-1}A_2$} (1.5,1);
    \draw[green!75!black] (-2.5,0) -- node[left] {\small $A_1^{-1} A_2^{-1}$}  (0,2.5);
    \draw[green!75!black] (0.5, -3)  --node[left] {\small $A_1 A_2^{-1}$} (-2.5,0);
    \draw[green!75!black] (0.5,-3) node[above]{$\gamma_3$};
    \draw[color=blue,fill=blue] (0.5,0) circle (0.5mm);
    \draw[color=green!75!black,fill=green!75!black] (0,2.5) circle (0.5mm);
    \draw[color=green!75!black,fill=green!75!black] (-2.5,0) circle (0.5mm);
    \draw[color=black,fill=black] (1.5,1) circle (0.5mm);
  \end{tikzpicture}
  \caption{The scattering diagram $\fD_{(2)}$ and the broken lines described in Example~\ref{brokenex}.} 
  \label{figbrokenex}
\end{figure}

\begin{example}
  Let us try one more calculation with broken lines. We take the same scattering diagram $\frakD_{(2)}$ in Figure \ref{fig:22case}. 
  Now take the initial exponent $\textbf{m}=(2,-2, -1,-1)$. 
  This time $\enpt$ would a small irrational perturbation of the point $(1, -1.5, 1 , -1.5)$ in the subspace $\{ (m,n) \in \wM | \, m = p^*(n) \}$.
  Note that after the projection $M \oplus N \rightarrow N$, the endpoint would still be  a small irrational perturbation of the point $(1, -1.5)$.
  By similar calculations we get 
  \[ 
    \vartheta _{\enpt, (2,-2, -1,-1)} = 
    A_1^2 A_2^{-2}X_1^{-1}X_2^{-1} + A_1^{-2}A_2^2 X_1^{1}X_2^{1}+ A_1^{-2}A_2^{-2}X_1^{-1}X_2^{1} + 2 A_2^{-2}X_1^{-1} + 2A_1^{-2}X_2.  
  \]
  We can again put $X_i =1$ and obtain $\vartheta^{\calA}_{\enpt, (2,-2)} = A_1^2 A_2^{-2} + A_1^{-2}A_2^2 + A_1^{-2}A_2^{-2} + 2 A_2^{-2} + 2A_1^{-2}$.
  Note that 
	\begin{equation} \label{eqn:notcc}
	\vartheta^{\calA}_{\enpt, (2,-2)} = 
    \left(\vartheta^{\calA}_{\enpt, (1,-1)}\right) ^2 -2. 
	\end{equation}
  Note that $\vartheta^{\calA}_{\enpt, (2,-2)}$ is not a cluster monomial. The above equation gives interesting property for theta functions. 
In this example, one can have a sense of how the $\calX$-theta functions are defined.
Consider the slice $\{(\textbf{m}, \textbf{n}) | \textbf{m} = p^*(\textbf{n}) \}$ in $\Dprin$. Note that by the choice of the endpoint $\enpt$, all the broken lines with initial slope $\textbf{m}=(2,-2, -1,-1) $ and endpoint $\enpt$ lie on the slice. Thus one can consider a change of variable $z^{(p^*(n), n)} \mapsto z^n$ then we get $\theta^{\calX}_{(-1,-1)} =  X_1^{-1}X_2^{-1} +  X_1^{1}X_2^{1}+ X_1^{-1}X_2^{1} + 2 X_1^{-1} + 2X_2$.
$\calX$ theta functions are regular functions on the $\calX$ cluster varieties.
For more details about $\calX$-theta functions, one should consult \cite{ghkk} or \cite{note}.
\end{example}

\section{The stability scattering diagram} \label{ch:chapter_motivic}

In this section, we will survey the link between stability conditions and scattering diagram due to Bridgeland in \cite{Bridge}. 
We will further interpret wall crossing in terms of Hall algebra multiplication defined in \eqref{eqn:hallmult}.

\subsection{Hall algebra} \label{sec:Hallalg}

We are going to state the definitions of the motivic Hall algebra developed by Joyce \cite{joyce2007configurations}, the groups $\hat{G}_{Hall}$, $\hat{G}_{reg}$, and the Lie algebras $\frakg_{Hall}$, $\frakg_{reg}$.
Naively speaking, elements of $\hat{G}_{Hall}$, $\frakg_{Hall}$ are stacks while 
elements of $\hat{G}_{reg}$, $\frakg_{reg}$ are varieties.
Then the stability scattering diagram is defined to take values in $\frakg_{Hall}$.
The reason for the lengthy definition is that we are working with stacks. 
However, by a deep theorem of Joyce (Theorem \ref{thm:joyce}), 
there exists an element in $\hat{G}_{reg}$ associated to semistable objects. 
Thus, we are actually working with $\frakg_{reg}$.

We would focus on acyclic quiver $Q$ from now on.
Set $N= \Z^{Q_0}$, $M = \Hom_{\Z} (N, \Z)$, $M_{\R} = M \otimes_{\Z} \R$.
Denote $N^{\oplus} = \{ n \in N | n_i \geq 0 \ \forall  i\}$.
Write $\C Q$ as the path algebra of $Q$. 
Let $\rep (Q) = \bmod \C Q$ the abelian category of finite dimensional representations of $Q$. Let $(e_i)_{i \in Q_0}$ be the canonical basis indexed by the vertices of $Q$.

To have a consistent story, we should consider the cluster algebras with principal coefficients, i.e. we consider the quiver $\wQ$ associated to $Q$, where $\wQ$ is defined by adding an frozen vertex $i'$ to a vertex $i$ in $Q$ and an arrow from $i'$ to $i$ (since we are considering opposite quiver for quiver representations). 
Then $\rep(Q)$ is a subcategory of $\rep(\wQ)$.
A representation $D$ of $Q$ with dimension vector \textbf{d}$\in N$ would then be represented as $(\textbf{d},\textbf{0}) \in \wN = N \oplus M$, where \textbf{0} is the zero vector with $n$ entries.

Define the algebraic stack $\calM$ parametrizing all objects of the category $\rep (Q)$ as a fibered category over the category of the schemes. The objects of $\calM$ over a scheme $S$ are pairs $(\calE, \rho)$ where $\calE$ is a locally free $\calO_S$-module of finite rank, and $\rho: \C Q \to \End_S (\calE)$ is an algebra homomorphism. Let $\St /\calM$ denote the full subcategory consisting of objects $f: X \to \calM$ for which $X$ is an algebraic stack of finite type over $\C$ and has affine stabilizers. 

\begin{definition}
The Grothendick group $K (\St/ \calM)$ of stacks over $\calM$ is the free abelian group with basis given by isomorphism classes of objects of $\St / \calM$, modulo the subgroup spanned by the relations
\begin{itemize}
	\item for every object $f: X \to \calM$ in $\St / \calM$, and every closed substack $Y \subset X$ with complementary open substack $U = X \setminus Y $, we have $[ X\xrightarrow{f} \calM] = [Y \xrightarrow{f|_Y} \calM ] + [ U \xrightarrow{f|_U} \calM]$;
	\item for every object in $\St / \calM$, and every pair of morphisms $h_1: Y_1 \to X$, $h_2: Y_2 \to X$ which are locally trivial fibrations in the Zariski topology with the same fibres, we have $[Y_1 \xrightarrow{ g \circ h_1} \calM ] = [Y_2 \xrightarrow{g \circ h_2} \calM]$. 
\end{itemize}

\end{definition}

$K(\St/ \calM)$ has the structure of a $K(\St / \C)$-module, defined by setting 
\[[X]\cdot [Y \xrightarrow{f} \calM] = [X \times Y \xrightarrow{f \circ \pi_2} \calM]\]
and extending linearly. There is a unique ring homomorphism 
\[
\Upsilon : K(\St / \C) \to \C (q)
\]
which takes the class of a smooth projective variety $X$ over $\C$ to the Poincar\'{e} polynomial 
\[
\Upsilon ([X]) = \sum_{i=0}^{2d} \dim _{\C} H^i (X_{an}, \C) \cdot q^{i/2} \in \C[q],
\]
where $X_{an}$ denotes $X$ considered as a smooth projective variety, and $H^i(X_{an}, \C)$ denotes singular cohomology. Note that we work with $\C (q)$, where $q=t^2$ in the setting of \cite{Bridge}.

Define the $\C (q)$ vector space 
\[K_{\Upsilon} (\St/ \calM) = K(\St/ \calM) \otimes_{K(\St/ \C)} \C (q)\]
with the relation $[X \times Y \xrightarrow{f \circ \pi_2} \calM ] = \Upsilon ([X]) \cdot [Y \xrightarrow{f} \calM]$. 


Next,
let us equip $K( \St / \calM)$ with a ring structure.
Denote $\calM^{(2)}$ be the stack of short exact sequences in $\rep (Q)$. The objects of $\calM^{(2)}$ over a scheme $S$ consist of three pairs $(\calE_i, \rho_i)$ of $\calM (S)$ for $i = 1,2,3$, together with morphisms $\alpha$ and $ \beta$ of $\calO_S$-modules which define a short exact sequence 
$0 \to \calE_1 \xrightarrow{ \alpha} \calE_2 \xrightarrow{ \beta} \calE_3 \to 0$.
There is a diagram
\[\begin{tikzcd}
	\calM^{(2)} \arrow[r, "b"] \arrow[d, "{(a_1, a_2)}"] & \calM \\
	\calM \times \calM & 
\end{tikzcd}\]
where 
\[b([0 \raw A_1 \raw B \raw A_2 \raw 0 ]) =B,\]
\[a_i ([0 \raw A_1 \raw B \raw A_2 \raw 0 ]) = A_i \text{, for }i =1,2. \]

Then the multiplication on $K(\St / \calM)$ is defined as
\begin{equation} \label{eqn:hallmult}
[ X_1 \xrightarrow{f_1} \calM ] \star [X_2 \xrightarrow{f_2} \calM ] =[ Z \xrightarrow{b \circ h} \calM],
\end{equation}
where $Z$ and $h$ are defined by
\[\begin{tikzcd}
Z              \arrow[d] \arrow[r, "h"]    & \calM^{(2)} \arrow[d, "{(a_1, a_2)}"] \arrow[r, "b"] & \calM \\
X_1 \times X_2 \arrow[r, "f_1 \times f_2"] & \calM \times \calM                               &
\end{tikzcd}\]
where the square is Cartesian.
Then $K(\St / \calM)$ becomes a ring. As the multiplication operation is $K( \St/ \C)$-linear,
it defines an algebra structure on the $\C (q) $-vector space $K_{\Upsilon}( \St / \calM)$. 
Then we get a $\C (q)$-algebra $H(Q)$. 

The stack $\calM$ can be decomposed as a disjoint union
\[\calM = \coprod_{d \in N^{\oplus}} \calM(d),\]
where $\calM(d)$ parameterize representations of $Q$ of the fixed dimension vector $d$. 
This induces a grading $H(Q) = \bigoplus_{d \in N^{\oplus}} H(Q)_d$, where $H(Q)_d = K_{\Upsilon} (\St / \calM_d)$. 
Note that 
\[H(Q)_0 = \C (q) \cdot 1 = [\calM_0 \subset \calM].\]
So we have a vector space decomposition $H(Q) = H(Q)_0 \oplus H(Q)_{>0}$, where $H(Q)_{>0} = \bigoplus_{d \in N^+} H(Q)_d$. 

We can further consider $H(Q)_{>k} = \oplus_{\dim(n) >k} H(Q)_n$ for each $k \in \N$. Then we have the quotient $H(Q)_{\leq k} = H(Q) / H(Q)_{> k}$. By taking limit, we obtain the completed algebra
\[
	\hat{H} (Q) = \underleftarrow{\lim} H_{\leq k}(Q).
\] 
Define $\frakg_{Hall} = H(Q)_{>0} $. Then we can obtain the corresponding completed Lie subalgebra
\[
	\hat{\frakg}_{Hall} = \hat{H}(Q)_{>0} \subset \hat{H}(Q)
\]
and the corresponding pro-unipotent group $\hat{G}_{Hall}$. 
By consider the embedding $\phi : \hat{G}_{Hall} \raw \hat{H}(Q)$, 
we can identify 
\[
	\hat{G}_{Hall} \cong 1 +\hat{H}(Q)_{>0} \subset \hat{H}(Q).
\]
Through this identification, $\hat{G}_{Hall}$ can act on $ \hat{H}(Q)$ by conjugation.

Objects in $H(Q)$ are of the form $f: X \to \calM$ such that $X$ is an algebraic stack of finite type over $\C$ and has affine stabilizers. Now we restrict $X$ to be a variety and consider the $\C[q, q^{-1}]$-submodule $H_{reg}(Q)$ generated by those objects.
By \cite[Theorem 5.2]{Bridge}, $H_{reg} (Q)$ is closed under the Hall algebra product and is then a $\C[q, q^{-1}]$-algebra.

By repeating the similar grading decomposition as above,
we obtain the Lie subalgebra $H_{reg} (Q)_{>0}$ of $H_{reg}(Q)$.
Define
\[
	\frakg_{reg} = (q-1)^{-1} H_{reg}(Q)_{>0} \subset \frakg_{Hall}.
\]
Then we take completion again and have the Lie algebra 
$ \hat{\frakg}_{reg} = (q-1)^{-1} \hat{H}_{reg}(Q)_{>0}$
and the corresponding pro-unipotent group $\hat{G}_{reg} \subset \hat{G}_{Hall}$. 

\subsection{Stability conditions and scattering diagrams} \label{sec:stability}

Now let us define the relevant notion of stability condition.

\begin{definition}
Given $\stab \in \wM_{\R}$, an object $C \in \rep(Q)$ is said to be \emph{$\stab$-semistable} if 
\begin{itemize}
	\item $\stab(C) = 0$, 
	\item every subobject $B \subset E$ satisfies $\stab(B) \leq 0$.
\end{itemize}
\end{definition}

By the definition of Grothendick group, we can identify elements in $K(\rep (Q))$ with its dimension vector,
 so we have $N \cong K(\rep (Q))$.
Now given a fixed $\stab \in M_{\R}$, we can define a stability function $Z: K(\rep (Q)) \to \C$ as $Z(C) = -\stab (C) + i \delta (C)$, where $\delta = (1, \dots, 1)$, i.e. $\delta (C) = \dim (C)$.
Note that if $C \neq 0$, $Z(C)$ lies on the upper half plane. So we can define the \emph{phase} of $C$ by 
\[\phi  (C) = \frac{1}{\pi} \arg Z(C).\]
Then an object $0 \neq C \in \rep (Q)$ is \emph{$Z$-semistable} if every nonzero subobject $B \subset C$ satisfies $\phi (B) \leq \phi (C)$.
Note that $0 \neq C \in \rep (Q)$ is $\stab$-semistable precisely if it is $Z$-semistable with phase 1/2.

For every $ 0 \neq C \in \rep (Q)$, there exists a Harder--Narasimhan filtration
$ 0 = C_0 \subset C_1 \subset \cdots \subset C_{k-1} \subset C_k =C$
whose factors $F_i = C_i / C_{i-1}$ are $Z$-semistable with descending phase: 
$\phi (F_1) > \phi (F_2) > \cdots > \phi (F_k)$. 

Now for any interval $I \subset (0,1)$ there is an open substack $\calM_I = \calM_I (\stab) \subset \calM$ parameterising representations of $Q$ lying in $\calP (I) \subset \rep (Q)$, where $\calP(I)$ is the full subcategory consisting of objects whose Harder Narasimhan factors have phases in $I$. This defines a corresponding element 
\[1_I (\stab) =[\calM_I (\stab) \subseteq \calM ] \in \hat{G}_{Hall} \subseteq \hat{H}(Q).\]
In particular, if $I = \{ 1/2 \}$, by the remark above, $0 \neq C \in \rep (Q)$ is $\stab$-semistable precisely if it is $Z$-semistable with phase $1/2$, then $1_{1/2} (\stab) = 1_{ss} (\stab)$. 

After all these set up, we can finally state a result of Bridgeland about the linkage between stability conditions and scattering diagrams.

\begin{theorem} \cite{Bridge} \label{thm:Hallscattering}
If $p^*$ is injective, there exists a consistent scattering diagram $\frakD$ in $M_{\R}$ such that: 
\begin{enumerate}
	\item the support $\frakd$ consists of maps $\stab \in M_{\R} = \Hom(N, \Z) \otimes \R$ for which there exist $\stab$-semistable objects in $\rep(Q)$; 
	\item the wall-crossing automorphism at a general point $\stab \in \frakd \subset \supp(\frakD)$ is 
$\Phi_{\frakD} (\frakd) = 1_{ss} (\stab) \in \hat{G}_{Hall}$.
\end{enumerate}
The scattering diagram is unique up to equivalence. It is called the \emph{stability scattering diagram}.

The stability scattering diagram is equivalence to the cluster scattering diagram in \cite{ghkk}.
\end{theorem}

Bridgeland further \cite{Bridge} shows that there is a bijection between equivalence classes of consistent $\hat{\mathfrak{g}}_{Hall}$-complexes and elements of the group $\hat{G}_{Hall}$. In this case, the stability scattering diagram $\frakD$ corresponds to $\Phi_{\frakD} = 1_{\rep(Q)} \in \hat{G}_{Hall}$. 
Let us recall a theorem of Joyce
\begin{theorem} \label{thm:joyce}
For any $\stab \in M_{\R}$ the element of $\hat{G} \subset \hat{H}(Q)$ defined by the inclusion of the open substack of $\stab$-semistable objects $\calM_{ss}(\stab) \subset \calM$ corresponds to an element $1_{ss}(\stab) \in \hat{G}_{reg}$.
\end{theorem}

Travis Mandel and the author have given a more explicit construction of stability scattering diagram in \cite{cheung2019donaldson} by using idea from \cite{GrossPandharipandeSiebert10}. 

One should note that the condition $p^*$ being injective is not `necessary'. We can again consider the cluster algebras with principal coefficients, i.e. we consider the quiver $\widetilde{Q}$ to any acyclic quiver $Q$, where $\widetilde{Q}$ is defined by adding an frozen vertex $i'$ to a vertex $i$ in $Q$ and an arrow from $i'$ to $i$ (since we are considering opposite quiver for quiver representations). 
We can consider $\rep(Q)$ as a subcategory of $\rep(\widetilde{Q})$ to construct the Hall algebra. 
In this setting, the scattering diagram $\frakD^{\calA_{\text{prin}}}$ would lie in $\wM$. 

\subsection{Wall crossing and Hall algebra theta functions} \label{sec:wallcrossing}
In this section, we will have a closer look at the wall crossing automorphism.
Before that, let us relate stability scattering diagram with torsion pair defined by $\stab$ in $\rep (Q)$.

\begin{definition} \label{def:torsion}
A \emph{torsion pair} in $\rep(Q)$ is defined to be a pair of full additive subcategories $(\calT, \calF)$ of $\rep(Q)$ such that
\begin{itemize}
\item if $T \in \calT$ and $F \in \calF$, then $\Hom_{\rep(Q)} (T, F) = 0$
\item for any $C \in \calA$ there is a short exact sequence
\[ 0 \raw T \raw C \raw F \raw 0 \]
with $T \in \calT$ and $F \in \calF$.
\end{itemize}
\end{definition}

A torsion pair always exists by the following theorem. 

\begin{lemma} \cite[Lemma 6.6]{Bridge} \label{thm:torsionpair}
For each $\stab \in \wM_{\R}$ there is a torsion pair $(\mathcal{T}({\stab}), \mathcal{F}({\stab})) \subset \rep(Q)$ defined by setting
\begin{align*}
\mathcal{T}(\stab) & = \mathcal{P} (1/2,1) \\
	& = \{ C \in \rep(Q): \text{any quotient object } C \rightarrow D \text{ satisfies } \stab(D)>0 \} \\
	\mathcal{F}(\stab)  	 & = \mathcal{P} (0,1/2] \\
	 &=  \{C \in \rep(Q) : \text{any subobject } F \subset C \text{ satisfies } \stab(F) \leq 0\}
\end{align*}
\end{lemma}

We will denote $1_{\calT} (\stab) := 1_{(1/2, 1)} (\stab)$, $1_{\calF} (\stab) := 1_{(0,1/2]} (\stab) \in \hat{G}_{Hall}$.
As $H(Q)$ is $N^{\oplus}$-graded, we can obtain an associated $N^{\oplus}$-graded algebra $H(Q) \otimes_{\C} \C[M]$ with relations 
\begin{equation} \label{eqn:comm}
z^{(\textbf{m}, \textbf{n})} \cdot H = q^{-\textbf{m} \cdot \textbf{d}} H \cdot z^{(\textbf{m}, \textbf{n})},
\end{equation}
 where $\textbf{m} \in M$, $H \in H(Q)_\textbf{d}$. The exponent of $q$ is $-\textbf{m} \cdot \textbf{d}$ since we are actually having 
 $-(\textbf{m}, \textbf{n}) \cdot (\textbf{d}, \textbf{0})$.
 By the same limiting process in Section \ref{sec:Hallalg}, we obtain the completion $\widehat{H(Q) \otimes_{\C} \C[M]}$. 

Next, we want to generalize the notions of broken line $\gamma$ to the stability scattering diagram. 
We consider the projection $\chi: H(Q)  \to \C$, 
\begin{equation} \label{eqn:euler_hall}
\chi ([X]) = \Upsilon ([X]) |_{q^{1/2}=-1} = \sum (-1)^i \dim _{\C} H^i (X_{an}, \C),
\end{equation}
where $X_{an}$ denotes $X$ as a smooth complex analytic variety. 

To each linear piece of $\gamma$, the attaching 
monomials attaching are replaced by $[X \to \calM] \otimes z^{(\textbf{m}, \textbf{n})}$.
Behavior at a bend at a wall $\frakd$ is given by conjugation by $\Phi_{\frakD}(\frakd)$. 
We can retain the ordinary attaching monomials by applying 
\[\chi :H(Q) \otimes_{\C} \C[\wM] \to \C[\wM], \]
\begin{equation} \label{eqn:chi1}
\chi([X \to \calM]) = \chi(X) z^{ (p^*(\dim(X)), \dim(X)) } .
\end{equation}

Noted above that on crossing a generic point of a wall $\frakd \in \frakD$ in the positive direction, the wall crossing automorphisms are given by 
\[ z^{(\textbf{m}, \textbf{n})} \mapsto \Phi_{\frakD} (\frakd) (z^{(\textbf{m}, \textbf{n})} ).\]
This is the conjugation of $\hat{G}_{reg}$ acting on $\widehat{H(Q) \otimes_{\C} \C[M]}$.
After taking $\chi$ to $\Phi_{\frakD} (\frakd) (z^{(\textbf{m}, \textbf{n})})$, we will obtain the path-ordered products as in \eqref{eqn:path_ordered}.

In particular, for an acyclic quiver, Bridgeland \cite[Lemma 11.4 and 11.5]{Bridge} has shown that the the stability scattering diagram is equivalent to the cluster scattering diagram after applying the Euler characteristic map.

Then for each $(\textbf{m}, \textbf{n})$ lies in $\calC$, where $\calC \in \Delta$ (cluster complex), we can define a \emph{Hall algebra theta function} as 
\begin{equation} \label{eqn:thetafun}
\vartheta_{(\textbf{m}, \textbf{n})} (\stab) = 1_{\calF} ((\textbf{m}, \textbf{n}))^{-1} z^{(\textbf{m}, \textbf{n})} 1_{\calF} ((\textbf{m}, \textbf{n}))
	\in \widehat{H(Q) \otimes_{\C} \C [M]},
\end{equation}
At the first sight, this definition seems to differ from the one in Section \ref{sec:broken}. 
At the same time, by Proposition \ref{rmk:thetaiscluster}, we can calculate theta functions considering paths to the positive chamber. Thus, the two definition coincide after applying Euler characteristic map to the Hall algebra theta function.
Note that here our direction goes from the negative chamber to the positive chamber which is the reverse of the set up in \cite{Bridge}.
After applying $\chi$ to this Hall algebra theta function, we will get back the theta function in usual sense. We will give more interpretation of theta function in Section \ref{sec:halltheta}. 

\subsection{Stability broken lines} \label{sec:def_hallline}
In this section, we will repeat similar ideas as in Section \ref{sec:broken} to define stability broken lines.
Consider $\frakD$ a stability scattering diagram. 
\begin{definition} \label{hall_brokendef}
Let $\frakD$ be a stability scattering diagram, $\textbf{m} \in \wM \setminus \{0\}$ and $\enpt \in \wM_{\R} \setminus \text{Supp}(\frakD)$. 
(For simplicity, we write $\textbf{m} \in \wM \setminus \{0\}$ instead of $(\textbf{m}, \textbf{n}))$ in this definition.)

A \emph{stability broken line} for $m$ with endpoint $\enpt$ is a piecewise linear continuous proper path $\gamma : ( - \infty , 0 ] \rightarrow M_{\mathbb{R}} \setminus Sing (\frakD)$ with a finite number of domains of linearity.
An element $[\calN \raw \calM] z^m \in H(Q) \otimes_{\C} \C[M]$,
 where $\calN \raw \calM$ factors through $\calN \raw \calM_d \raw \calM$, with $d$ the dimension vector for the representation,
 is attached to each domain of linearity $L \subseteq ( - \infty, 0)$ of $\gamma$.
 The path $\gamma$ and the $[\calN \raw \calM] z^{\textbf{m}}$ need to satisfy the following conditions:
\begin{itemize}
    \item $\gamma(0) = \enpt$.
    \item If $L$ is the first (i.e., unbounded) domain of linearity of $\gamma$, then $[\calN \raw \calM] z^\textbf{m} = z^{\textbf{m}}$. This is corresponding to the zero representation $[\bullet \mapsto 0]$.
    \item In each domain of linearity $L \subset (-\infty, 0]$, the attached Hall algebra monomial is $[\calN \subset \calM] z^m$. After applying integration map, by \eqref{eqn:chi1}, we have
    \[ \chi ([\calN \subset \calM])z^\textbf{m} = \chi ([\calN]) z^{\wpp^*(\dim \calN)+\textbf{m}}.
    \]
    Then for $t \in L$, $\gamma'(t) = -(p^*(\dim \calN)+m)$.
    \item $\gamma$ bends only when it crosses a wall.
     If $\gamma$ bends from the domain of linearity $L$ to $ L'$ when crossing $\frakd$ defined in Theorem \ref{thm:Hallscattering}, then $[\calN \subset \calM] z^m$ is a term in $\Phi_{\frakD} (\frakd) \cdot ([\calN \subset \calM] z^\textbf{m})$.
\end{itemize}
\end{definition}

\subsection{Further properties in the stability scattering diagrams} \label{sec:prop_hall}
In later sections, we will compute the Hall algebra wall crossing explicitly. Let us formulate some equations in the Hall algebra in this section.

In Section \ref{sec:stability}, for $\stab \in M_{\R}$, we define 
\[
	1_{ss} (\stab) = [\calM_{1/2} (\stab) \subseteq \calM ],
\]

First note that if $\stab$ is a general point on some wall $\frakd$ in the cluster complex, there is a representation $D$ with $\dim D$ primitive which is a $\stab$-semistable object. 
 We are going to show $D$ is indecomposable. 
 If not, $D= D_1 \oplus D_2$ for some $D_1$, $D_2$ non zero. 
 Then for all $\stab \in \frakd_i$, $\stab(D_1)  + \stab(D_2)=\stab (D) =0$.
 As $D_1$ and $D_2$ are subojects of $D$, by definition of semistability,
 $\stab(D_1), \stab(D_2) \leq 0$.
 Therefore, $\stab(D_1)= \stab(D_2) = 0$.
 However, $\frakd$ is of co-dimension $1$ in $N$.
 Hence the normal space to $\frakd$ is of dimension $1$ only.
 This implies $\dim D_1$ and $\dim D_2$ are proportional to $\dim D$, contradicting to the condition that $\dim D$ is primitive.
 Thus $D$ is indecomposable. 
 
 However, note that $D$ may not be the unique indecomposable representation which is $\stab$ semistable.
 Consider a regular representation $D$. For representation $K$ with dimension vector $k \dim D $, where $k $ is a positive integer, the representation $K$ may still be indecomposable. 
 For this case, we would still have the primitive normal of the wall is an indecomposable representation. It is only that scalar multiple of the primitive normal may correspond to indecomposable representation as well. 
 
If $D$ is preprojective or preinjective, then representation with dimension vector a multiple of $\dim D$ would not be indecomposable. 
Hence for $\stab$ such that it corresponds to a preprojective or preinjective representation, $1_{ss} (\stab)$ is a formal sum 
 \begin{equation} \label{eqn:1ss}
 1_{ss} (\stab) = 1 +  \sum_{k \geq 1} [BGL_k (D) \raw \calM],
 \end{equation}
 where $BGL_k(D)$ are classifying space for $GL_k$.
 More precisely, $[BGL_k(D) \raw \calM] $ means that a point is mapped to $k$ copies of D, i.e. $ \bullet \mapsto D^{\oplus k}$.

Theorem \ref{thm:Hallscattering} tells us that the wall crossing automorphism at $\stab$ is conjugation by $1_{ss} (\stab)$.
 Therefore, we wish to understand $1_{ss}(\stab)^{-1}$ which is
 \begin{equation} \label{eqn:inverse}
 1_{ss}(\stab)^{-1}=1 + \sum_k (-1)^k \prod_{l=1}^k [BGL_{r_l}(D) \raw \calM],
 \end{equation}
 where the product means multiplying $k$ many $[BGL_{r_l} \raw \calM ]$, for some $r_l \in N$, together. 
 For example, up to degree 2,
we have
 \begin{align*} 
  1_{ss}(\stab)^{-1} & \sim 1 - [BGL_1 (D)\raw \calM] - [BGL_2 (D)\raw \calM] \\
 +& [BGL_1(D) \raw \calM]^2 + [BGL_1(D) \raw \calM] \star [BGL_2 (D)\raw \calM] \\
  + &[BGL_2 (D)\raw \calM] \star [BGL_1(D) \raw \calM].
 \end{align*}

Let us also recall how to express $\GL_d$ as an element in $K(\Var / \C)$:

\begin{lemma} \cite[Lemma 2.6]{bridge_hall}
\begin{align*}
[\GL_d] & = [\Aff^{ d(d-1)/2}] \prod_{k=1}^d (\Aff^k -1).
\end{align*}
\end{lemma}

By similar techniques, we can also calculate the product $\prod_{l=1}^k [BGL_{r_l}(D) \raw \calM]$.
 As the product is associative, we can first work on $[BGL_{r_1}(D) \raw \calM] * [BGL_{r_2}(D) \raw \calM]$.
 This consists of exact sequences
 \[
0 \longrightarrow D^{\oplus r_1} \longrightarrow Y \longrightarrow D^{\oplus r_2} \longrightarrow 0,
 \]
 where the automorphisms of the exact sequence can be viewed as `upper triangular' matrices of the form
 \[\left( 
 \begin{array}{c c}
 \GL_{l_1} & * \\
 0 & \GL_{l_2} 
 \end{array}
 \right),\]
 where $*$ denotes entries with arbitrary element in $\C$.
 Thus in general, the underlying stack of \[\prod_{l=1}^k [BGL_{r_l}(D) \raw \calM]\]
 is of the form
$BG$ where $G$ is the group of $(r_1 + \cdots + r_k)$ block-upper triangular matrices:
 \[\left(
 \begin{array}{c c c c}
 \GL_{l_1} & * & * & * \\
 0 & \GL_{l_2} & * & * \\
 0 & 0 & \ddots & * \\
 0 & 0 & 0 & \GL_{l_k}
 \end{array}
 \right).\]
 Hence after applying the integration map to $\prod_{l=1}^k [BGL_{r_l}(D) \raw \calM]$, we have
 \begin{align*}
  & \chi(\prod_{l=1}^k [BGL_{r_l}(D) \raw \calM]) \\
  =& \frac{1}{ \left( \prod_{l=1}^k q^{r_l(r_l-1)/2} \prod_{s=1}^{r_l} (q^s-1) \right)  \prod_{u<v} q^{r_u r_v} }.
 \end{align*}

\section{Theta functions and the Caldero-Chapoton formula}

In the finite classification of cluster algebras, Fomin and Zelevinsky have shown \cite{cluster2} the cluster algebras of finite type can be associated to a Dynkin quiver.
At the same time, the set of indecomposable representations of a
quiver of Dynkin type is in bijection with the set of positive roots by Gabriel's theorem.
Using these correspondence, 
Caldero and Chapoton \cite{ccformula} then related cluster variables with indecomposable quiver representations for ADE type cluster algebras. 
They found that the coefficients of the cluster variables are the Euler characteristics of the quiver Grassmannians. This proves that the coefficients are non-negative.
Later Caldero-Keller \cite{caldero2008triangulated} generalized the result to acyclic quivers.

Let $Q$ be an acyclic finite quiver with vertices $1, \dots, n$. 
Let $D$ be a finite-dimensional representation of $Q$ with dimension vector $\textbf{d}$.
Denote $Gr_\textbf{c}(D) :=\{ C \in \bmod (Q) | C \subseteq D, \dim (C) =\textbf{c} \}$ for $\textbf{c} \in \N^n$.
Define the Caldero-Chapoton (or the cluster character) formula as
\[CC(D) = \frac{1}{A_1^{d_1} \cdots A_n^{d_n}} \sum_{0 \leq \textbf{c} \leq \textbf{d}} \chi (Gr_\textbf{c}(D)) \prod_{i=1}^n A_i^{\sum_{j \rightarrow i} c_j +\sum_{i \rightarrow j} (d_j-c_j)}, \]
where the sum is taken over all vectors $\textbf{c} \in \N^n$ such that $0 \leq c_i\leq d_i$ for all $i$.
Then if the quiver $Q$ is of type ADE, Caldero-Chapoton showed that
\begin{theorem} \label{thm:cc} \cite{ccformula}
 \[CC(D ) = X_D,\]
 where $X_D$ is the cluster variable obtained from $D$ by composing Fomin-Zelevinsky's bijection with Gabriel's.
\end{theorem} 

We have
\begin{align*}
CC(D)
	& = \frac{\prod A_i^{\sum_{i \rightarrow j} d_j}}{A_1^{d_1} \cdots A_n^{d_n}} \sum_{0 \leq \textbf{c} \leq \textbf{d}} \chi (Gr_\textbf{c}(D)) \prod_{i=1}^n A_i^{\sum_{j \rightarrow i} c_j +\sum_{i \rightarrow j} -c_j} 
\numberthis \label{eqn:newcc}\\
	& = z^{-\calE(\textbf{d})} \sum_{0 \leq \textbf{c} \leq \textbf{d}} \chi (Gr_\textbf{c}(D)) z^{p^*(\textbf{c})},
\end{align*}
where $\calE(d) $ is defined as in \eqref{eqn:calE}.

\begin{remark} \label{sec:scatt_camber}
In this remark, we will combine the discussion in remark \ref{rk:cg} and Section \ref{sec:prop_hall} to talk about the roles of the $c$ and $g$ vectors in terms of the chambers of scattering digrams and quiver representations.

First, let us re-write the cluster character map in terms of principal coefficients. 
\[
CC(D) = z^{-(\calE(\textbf{d}),0)} \sum_{0 \leq \textbf{c} \leq \textbf{d}} \chi (Gr(\textbf{c},D)) z^{\wpp^*(\textbf{c},0)}. 
\]

Combining with \eqref{eq:gvector}, we learnt that $-\calE(\textbf{d})$ are $g$ vectors of the corresponding cluster monomials. 

Let $\calC_\prin$ denote a chamber in the cluster complex of the scattering diagram $\frakD^{\cAp}$.
Consider the projection map $M_\R \oplus N_\R \rightarrow M_\R$. 
The image of $\frakD^{\cAp}$ would be a scattering diagram $\frakD$ in $M_R$.
We will denote the image of $\calC_\prin$ as $\calC$. Note that $\calC$ is a chamber in the cluster complex of $\frakD$ and it is a maximal cone of a simplicial fan by \cite[Theorem 2.13]{ghkk}.
Thus we can consider $\g_1, \dots, \g_n$ the generators of $\calC$, and $\bc_1, \dots,  \bc_\n$ the normal vectors of the walls bounding the chamber $\calC$. 

In remark \ref{rk:cg}, we learnt that the normal vectors are the $c$ vectors of the seed \seed ~corresponding to the chamber. 
Next we comibine the idea in the beginning of Section \ref{sec:prop_hall}.
There we associate the primitive normal vector of the walls to the indecomposable representation of Q. 
In particular, for $\calC$ to be in the cluster complex, we can deduce the $c$ vectors are actually the dimension vectors of the indecomposable pre-projective or pre-injective representations of $Q$.

Next we turn our attention to the $g$-vectors.
In remark \ref{rk:inj}, the $g$ vectors actually represent the injective resolution of the quiver representation $D$. 
Hence if $g_k$ is not the standard basis vectors, 
then there is an indecomposable representation $D_k$ of $Q$ such that $-\calE(D_k) = \textbf{g}_k$.
If $g_k$ is one of the standard basis vectors $e_i$, we can associate $g_k$ with the shift of injective $I(i)[-1]$ from the cluster category theory in \cite{BMRRT}, \cite{keller_cat}.  
The above discussion is known to representation theorists by tilting theory, 
here we just want to simply draw this association by looking the dimension vectors. 
\end{remark} 

Let us consider some examples:
\begin{example}
Going back to Example \ref{ex:11case}, we have
  \[ 
    \vartheta_{\enpt, (1,-1)} =  
    A_1 A_2^{-1} +A_1^{-1} A_2^{-1}X_2  +A_1^{-1} A_2X_1X_2 .  
  \]
We can write it as
  \[ 
    \vartheta_{\enpt, (1,-1)} =  
    \frac{A_1^2}{A_1^{-1}A_2^{-1}}(1+ A_1^2 X_2+A_1^2 A_2^2X_1X_2) = CC(\C \rraws \C).
  \]
The three terms corresponds to the three representations with dimension vectors $(0,0)$, $(0,1)$, $(1,1)$.
\end{example}

For example, consider Figure \ref{fig:22case}, the scattering diagram associated to the Kronecker 2-quiver. Let us consider the chamber $\calC$ spanned by $(2,-1)$ and $(3, -2)$. Then the corresponding cluster is $\vartheta_{(2,-1)}, \vartheta_{(3,-2)}$ which is associated to the indecomposable projectives $0 \rraws \C$, $\C \rraws \C^2$.
 For any point $r$ in the chamber $\calC$, we can write $r = \lambda_1 (2,-1) + \lambda_2 (3, -2)$. 
The corresponding quiver representation of $r$ would be $ ( 0 \rraws \C)^{\lambda_1} + (\C \rraws \C^2)^{\lambda_2 }$.

\section{Positive Crossing and the AR quiver} \label{sec:AR}
Consider a path $\gamma$ in $\frakD$.
 Assume $\gamma$ crosses a wall $\frakd$.
 Let $\bc \in N^+$ be the normal of $\frakd$, i.e. $\frakd \subset \bc^{\perp}$. 
Then the crossing is a \emph{positive crossing} if the path passes from the region where $n$ is negative to the region it is positive. Otherwise, it is called a \emph{negative crossing}.

Note that in the two-dimensional scattering diagram, if the path travels around the origin anti-clockwise from the lower half plane towards the positive chamber, the crossing is always positive.

Now consider two walls $\frakd_1 $ and $\frakd_2$ in the cluster complex.
 Let $\textbf{c}_i \in N^+$ be the normal of $\frakd_i$ respectively.
 Then $\frakd_1$ and $\frakd_2$ may or may not intersect along a co-dimensional 2 polyhedral set which we call a \textbf{joint}.
 If $\frakd_1$ intersects $\frakd_2$ along a joint $\frakj$,
 we decompose $\frakd_1$ into three disjoint components $\frakd_1^+$, $\frakj$, $\frakd_1^-$ where each of them are connected.
 Let us assume $-p^*(\textbf{c}_1)$ does not lie on the joint $\frakj$.
 Then $\frakd_1^+$ is defined to be the connected piece which contains the point $-p^*(\textbf{c}_1)$ which we will name the outgoing piece of $\frakd_1$.
 Figure \ref{fig:gamma} illustrates the decomposition.
 If $\frakd_1$, $\frakd_2$ do not intersect along a joint,
 then we take $\frakd_1^+ = \frakd_1$.
 Note that $\frakd_1^+$ still contains the point $-p^*(\textbf{c}_1)$ by construction.
  
Next consider a path $\gamma$ has positive crossing from the outgoing piece of $\frakd_1$ to the wall $\frakd_2$. Please refer to Figure \ref{fig:gamma} for illustration.
 
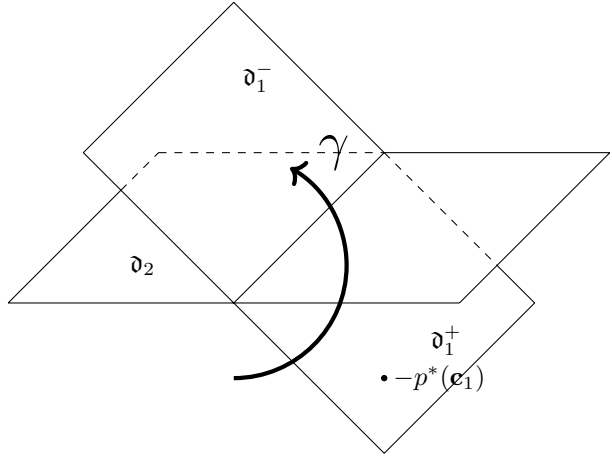
\begin{figure}
\centering
\begin{tikzpicture}
\draw
(1,1) -- (-1,-1)
(-4,-1) -- (-2.5, 0.5) 
(-4,-1) -- (2,-1)
(-1,3) -- (1,1)
(-3,1)-- (-1,3)
(-3,1) -- (1,-3)
(1,1) -- (4,1)
(4,1) -- (2,-1)
(1,-3) -- (3,-1)
(2.5, -0.5) -- (3,-1)
(-1,2) node[right] {$\frakd_1^-$}
(-2.5, -0.5) node[right] {$\frakd_2$}
(1.5,-1.5) node[right] {$\frakd_1^+$}
(0,1) node[right] {\huge $ \gamma$   }
(1,-2) node[right] {$-p^* (\textbf{c}_1)$}
;

\filldraw (1,-2) circle (1pt);

\draw [dashed]
(-2,1) -- (1,1)
(-2.5, 0.5) -- (-2,1)
(1,1) -- (2.5, -0.5)
;

\draw[ultra thick, ->] (-1,-2) arc (-90:60:1.5) 
;
\end{tikzpicture}
\caption{$\gamma$ goes from $\frakd_1^+$ to $\frakd_2$} \label{fig:gamma}
\end{figure} 
 
From last section, for $\textbf{c}_i \in N^+$, $i=1,2$, we can consider $\textbf{c}_i$ as a dimension vector for some indecomposable semistable representation $C_i$.
 Then since $\gamma$ having positive crossing from outgoing piece of $\frakd_1$ to $\frakd_2$, 
 $-p^*(\textbf{c}_1)$ points away from the positive direction of $\textbf{c}_2$, i.e.
\[
	\langle -p^*(\textbf{c}_1) , \textbf{c}_2 \rangle <0 .
\]
Here the inequality is strict as we have assume $-p^*(\textbf{c}_1)$ does not lie on $\frakd_2$.
We have
\begin{align*}
& \langle -p^*(\textbf{c}_1) , \textbf{c}_2 \rangle = - \{ \textbf{c}_1, \textbf{c}_2 \} \\
=& \chi(\textbf{c}_1, \textbf{c}_2) - \chi( \textbf{c}_2, \textbf{c}_1) \\
=& \dim(\Hom (C_1, C_2) ) - \dim \Ext^1 (C_1, C_2) \\
	& - \dim ( \Hom (C_2, C_1) ) + \dim (\Ext^1 (C_2, C_1)). \numberthis \label{eqn:ar}
\end{align*}

We will now apply the lemmas in Section \ref{rmk:ext_zero}.
 
First assume $Q$ is a connected acyclic quiver of finite type.
 Then from Theorem \ref{thm:AR_component}, we have the AR quiver is connected and it equals both the pre-projective component and the pre-injective component.
 At the same time the regular component is empty.
 Note that if $\dim(\Hom (C_1, C_2) ) \neq 0$, then $C_1$ is a predecessor of $C_2$ which implies $\dim ( \Hom (C_2, C_1) ) = 0$.
 From Section \ref{rmk:ext_zero}, we know that only one of the terms $\dim(\Hom (C_i, C_j) )$, $\dim \Ext^1 (C_i, C_j)$, $i \neq j$ is nonzero.
 In order to attain $\langle -p^*(\textbf{c}_1) , \textbf{c}_2 \rangle <0$, we will have
 \[ \langle -p^*(\textbf{c}_1) , \textbf{c}_2 \rangle =  - \dim \Ext^1 (C_1, C_2)- \dim ( \Hom (C_2, C_1) ).\]
 As the term above is negative, one of the $\dim \Ext^1 (C_1, C_2)$ and / or $\dim ( \Hom (C_2, C_1) )$ is non-zero which implies $C_2$ is a predecessor of $C_1$ in the AR quiver by Lemma \ref{thm:extnonzero} and Theorem \ref{thm:ARprop}. 

Now consider $Q$ is a connected acyclic quiver with infinitely many indecomposables. 
 Then the AR quiver consists of three components: pre-projective $\calP$, regular $\calR$ and pre-injective $\calI$.
 Assume $C_i$ lies in one of the $\calP, \calR, \calI$, for $i=1,2$ and $C_1 , C_2$ do not lie in the same component.
 We are going to see that $\gamma$ travels a direction from $\calI$ to $\calR$, to $\calP$.
 The reason is that from Theorem \ref{thm:PRI}, we know that
 \[ \Hom (\calR, \calP) = \Hom (\calI, \calP) = \Hom ( \calI , \calR) = 0.\]  
 And we also have 
 \[\Ext^1(\calP, \calR) = \Ext^1( \calP, \calI) = \Ext^1 (\calR, \calI) = 0\]
from Lemma \ref{thm:lemmaPRI}.
Therefore by using the same argument as above, we know that
 \[ \langle -p^*(\textbf{c}_1) , \textbf{c}_2 \rangle =  - \dim \Ext^1 (C_1, C_2)- \dim ( \Hom (C_2, C_1) ).\]
Hence we can only have three choices for $C_1$, $C_2$:
 $(1)$. $C_2 \in \calI$, $C_1 \in \calP$;  
 $(2)$. $C_2 \in \calI$, $C_1 \in \calR$;
 $(3)$. $C_2 \in \calR$, $C_1 \in \calP$.
 All three cases indicates that $C_2$ is the predecessor of $C_1$ in the general definition for predecessor. 
 This is saying that
 if $\gamma$ travels between walls corresponding to $\calP, \calR, \calI$,
 it would have crossed from the pre-injective to the regular then to the pre-projective if the crossings are positive.
 
We also have the cases $C_1, C_2 \in \calP $ or $C_1, C_2 \in \calI$.
 Then by Lemma \ref{thm:lemmaext}, Lemma \ref{thm:extnonzero} together with Theorem \ref{thm:ARprop},
 we have $C_2$ is a predecessor of $C_1$ from \eqref{eqn:ar}.
 Therefore, we can say if $\gamma$ is having a positive crossing in the scattering diagram, it is going in an opposite direction to the AR quiver. 
 Thus we have finished the proof of Theorem \ref{thm:ar} which is stated more precisely in the following proposition: 
 
\begin{theorem} \label{thm:arquiver}
Given two walls $\frakd_1$ and $\frakd_2$ in the cluster complex in $\frakD$.
 Let $\textbf{c}_i \in N^+$ be the normal of $\frakd_i$ for $i=1,2$.
 If the two walls intersect along a joint, a co-dimensional 2 polyhedral set,
 we consider the outgoing piece $\frakd_1^+$ of $\frakd_1$ which contains $-p^*(\textbf{c}_1)$ instead of the entire $\frakd_1$.

Let $\gamma$ be a path crossing the outgoing piece of $\frakd_1 \in \frakD$ and then $\frakd_2 \in \frakD$ with both crossings positive.
  Let $C_i$ be the corresponding indecomposable representation for $\frakd_i$, $i=1,2$.
  Then we have $C_2$ is a predecessor of $C_1$ in the Auslander-Reiten quiver of $Q$.
\end{theorem}

\section{Hall Algebra Theta function} \label{sec:halltheta}

In this section, we are going to give a generalization of the Caldero-Chapoton formula by the machinery set up in Chapter \ref{ch:chapter_motivic}. More precisely, we are going to show Theorem \ref{thm:theta} in the introduction which is giving an explicit description of the Hall algebra theta functions defined in \eqref{eqn:thetafun}.

Now let us recall Theorem \ref{thm:theta}:

\begin{theorem} \label{thm:sectheta}
Fix a stability scattering diagram $\frakD$.
Let $\enpt$ be a point in the positive chamber of $\frakD$.
Consider a point $\textbf{m}_0$ in the cluster complex.
We further assume that $\textbf{m}_0$ lies in the cluster chamber $\calC$ which corresponds to a seed not containing any initial cluster variable\footnote{The condition `the seed not containing any initial cluster variable' seems to be abrupt. This is only because the initial variables do not mutate or they correspond to the shift of injectives that the variables do not correspond to the indecomposable representations.}.
Then we can write $\textbf{m}_0 = -\calE (\textbf{d})$ for some $\textbf{d} \in (\Z_{\geq 0})^n$, where $\calE$ defined in \eqref{eq:gg} and we are going to see it is the $g$ vector. Then the Hall algebra theta function is
\[
\vartheta_{(\textbf{m}_0,0)} = \chi (\calG_{\calF (m_0)} (D))  z^{(\textbf{m}_0,0)},
\]
where objects in $\calG_{\calF (m_0)} (D)$ are representations $C$ in $\calF (m_0)$ equipped with an inclusion into $D$.
 By applying the Euler characteristic map $\chi$ (defined in \eqref{eqn:chi1}),
we get
\[
	\vartheta_{(\textbf{m}_0,0)} 
	 = z^{(- \calE (D),0)} \sum_{0 \leq \bc \leq \textbf{d}} \chi (\Gr(c,D)) z^{(p^*(\textbf{\bc}), \textbf{\bc})}.
\]
which is the cluster character formula.
 \end{theorem}

\begin{proof}

Recall in Lemma \ref{thm:torsionpair}:
\[ \calF (m_0) = \{ E \in \rep(Q) : \text{any subobject } F \subset E \Rightarrow  m_0 (F) \leq 0 \}. \] 
For any $C \subset D$, we have the inclusion $C \raw D$.
 Then $\Hom (C, D) \neq 0$.
 By decomposing $C$ into sum of indecomposables $C_i$,
 then $\dim \Hom (C_i, D) > 0$.
 This implies $C_i$ is a predecessor of $D$.
 By Lemma \ref{thm:lemmaext} and Lemma \ref{thm:lemmaPRI},
 we have $Ext^1 (C_i, D)=0$.
 Thus $\chi(\textbf{c}_i, \textbf{d}) = \dim \Hom (C_i, D) >0$.
 That is $m_0 (C) = - \calE(\textbf{d}) (\textbf{c}) <0$.
 Therefore we have $D \in \calF (m_0)$. 

Furthermore, for all $C \in \calF(m_0)$ indecomposable, by definition of $\calF(m_0)$,
 $\calE(d) (C) = \chi (\textbf{c},\textbf{d}) \geq 0$, where $\textbf{c} = \dim C$.
 By the discussion at the end of Section \ref{rmk:ext_zero},
 we obtain $\chi (\textbf{c},\textbf{d}) =\dim \Hom (C,D)$.

Recall the theta function we defined in \eqref{eqn:thetafun}
\[ \vartheta_{(\textbf{m}_0,0)}(\stab) = 1_{\calF} (\textbf{m}_0)^{-1} z^{(\textbf{m}_0,0)} 1_{\calF} (\textbf{m}_0).\]
Note that by the commutativity relation \eqref{eqn:comm}, we have
\[
	z^{(\textbf{m}_0,0)} [\calM^{\calF}_\textbf{c} \raw \calM ] = q^{ \calE(D) (\textbf{c})} [\calM^{\calF}_\textbf{c} \raw \calM ] z^{(\textbf{m}_0,0)}, 
\]
where objects in $[\calM^{\calF}_\textbf{c} \raw \calM ]$ are representations $C \in \calF (\textbf{m}_0)$ with dimension vector $\textbf{c}$. As noted above,
\[
	z^{(\textbf{m}_0,0)} [\calM^{\calF}_\textbf{c} \raw \calM ] = q^{ \dim \Hom (C,D)} [\calM^{\calF}_\textbf{c} \raw \calM ] z^{(\textbf{m}_0,0)}, 
\]
By viewing $q^{ \dim \Hom (C,D)} $ as $[\Aff^{\dim \Hom (C,D)}]$ in the Hall algebra, we have
\[
	 q^{\dim \Hom (C,D)} [\calM^{\calF}_\textbf{c} \raw \calM ] = [\calM^{\calF,d}_\textbf{c} \raw \calM],
\]
where objects in $[\calM^{\calF,d}_\textbf{c} \raw \calM]$ are pairs $(C , \psi)$ with $C\in \calF (m_0)$ having dimension vector $e$ and $\psi: C \raw D$ a map.

Now consider $(C , \psi)$ as an object in $[\calM^{\calF,d}_\textbf{c} \raw \calM]$.
As $C \in \calF (m_0)$, $R = \ker ( \psi ) \in \calF (m_0)$. 
Consider the short exact sequence
\begin{equation} \label{ses:theta}
0 \longrightarrow R \longrightarrow C \longrightarrow S \longrightarrow 0.
\end{equation}

The map $\psi$ induces an injective map $S \raw D$.
We wish to show $S \in \calF(m_0)$. 
Note that by the properties of torsion pair in Definition \ref{def:torsion}, 
we have the following exact sequence
\[
	0 \longrightarrow T \longrightarrow S \longrightarrow F \longrightarrow 0,
\]
where $T \in \calT(m_0)$, $F \in \calF(m_0)$. 
Since $D \in \calF(m_0)$, $\Hom (T, D) = 0$ from the definition of torsion pair.
So the exact sequence becomes
\[\begin{tikzcd}
 0 \arrow[r] & T  \arrow[r] \arrow[dr, "0"] &  S   \arrow[r] \arrow[d, hook] & F \arrow[r]  & 0 \\
             & &D &  &    
\end{tikzcd}\]
This forces the map $T \raw S$ to be the zero map in the exact sequence. Thus we have $S \cong F \in \calF(m_0)$. 

Going back to \eqref{ses:theta}, we have
\[
	[\calM^{\calF,d}_\textbf{c} \raw \calM] = 1_{\calF} (m_0) * \calG_{\calF (m_0)} (D),
\] 
where objects in $\calG_{\calF (m_0)} (D)$ are representations $C$ in $\calF (m_0)$ equipped with an inclusion into $D$. Then 
\begin{align*}
\vartheta_{m_0} 
 &=  1_{\calF} (m_0)^{-1} z^{(\textbf{m}_0,0)} 1_{\calF} (m_0) \\
 &= 1_{\calF} (m_0)^{-1} \sum_{e = \dim E} q^{\dim \Hom (E,D)} [\calM^{\calF}_e \raw \calM ] z^{(\textbf{m}_0,0)} \\
 & = 1_{\calF}(m_0)^{-1} * 1_{\calF} (m_0) * \calG_{\calF (m_0)} (D)  z^{(\textbf{m}_0,0)} \\
 & = \calG_{\calF (m_0)} (D)  z^{(\textbf{m}_0,0)} 
\end{align*}

 By applying $\chi$ in \ref{eqn:chi1} to $\calG_{\calF (m_0)} (D)$,
we get
\[
	\vartheta_{(\textbf{m}_0,0)} = \chi (\calG_{\calF (m_0)} (D))  z^{(\textbf{m}_0,0)}
	 = z^{(- \calE (D),0)} \sum_{0 \leq \textbf{c} \leq \textbf{d}} \chi (\Gr(c,D)) z^{(p^*(\textbf{c}), \textbf{c})}.
\]
which is exactly as in \eqref{eqn:newcc}.
 Therefore, we obtain the Caldero-Chapoton formula.
 
\end{proof}

\section{Stratification of quiver Grassmannian} \label{sec:strata}
We will repeat similar calculations as in the last section to give a stratification of quiver Grassmannian.

Consider an indecomposable quiver representation $D$ which is not regular.
 Then $m = -\calE (d)$ lies in the cluster complex.
 Now consider a broken line $\gamma : (\infty  ,0 ] \raw M_{\R} \setminus \{ 0 \}$ with endpoint $\enpt$ in the positive chamber and initial slope $-\calE(\textbf{d})$.
 We further assume the final slope of the broken line is  $ -\calE(\textbf{d}) + p^*(\textbf{c})$ for some $\textbf{c}$.
 Let $\gamma$ bend over the walls $\frakd_1, \dots , \frakd_s$ in the cluster complex and assume all the bendings are good.
 Denote by $\textbf{c}_i \in N^+$ the normal vectors of $\frakd_i$ for $i = 1, \dots s$.
 From the remarks in last sections,
 the walls $\frakd_i$ correspond to indecomposable representations $C_i$ with dimension vector $\textbf{c}_i$.

For $i=1, \dots , s-1$,
 if $\frakd_i$ intersect $\frakd_{i+1}$ along a co-dimension 2 joint $\frakj$ as in Section \ref{sec:AR},
 we decompose $\frakd_i$ into 3 connected components $\frakd_i^+$, $\frakj$ and $\frakd_i^-$ where $-p^*(\textbf{c}_i)$ lie on either $\frakd_i^+$ or $\frakj$.
 Let us further assume that $\gamma$ crosses from $\frakd_i^+$ to $\frakd_{i+1}$.
 If $\frakd_i$ does not intersect $\frakd_{i+1}$ along a co-dimensional 2 joint,
 we can take $\frakd_i^+ = \frakd_i$.
 By repeating same argument as in Section \ref{sec:AR}, we will have the following 2 cases.
 \begin{description}
 	\item[A] If $-p^*(\textbf{c}_i)$ lies on $\frakj$, then $\langle -p^*(\textbf{c}_i) ,\bc_{i+1} \rangle = - \{\bc_i,\bc_{i+1} \} =0$. We then have 
	\begin{align*}
	 	&\dim \Hom (C_i, C_{i+1}) = \dim \Hom (C_{i+1}, C_i)  \\
 	=&\dim \Ext^1 (C_i, C_{i+1}) =\dim \Ext^1 (C_{i+1}, C_i) = 0.
	\end{align*}	 	

 	\item[B] If $-p^*(\textbf{c}_i)$ lies on $\frakd_i^+$,
 	 and as $\gamma$ crosses good from $\frakd_i^+$ to $\frakd_{i+1}$, then 
 	 \[\langle -p^*(\textbf{c}_i) ,\bc_{i+1} \rangle = - \{\bc_i,\bc_{i+1} \} < 0.\]
 	 Therefore, \[\langle -p^*(\textbf{c}_i) ,\bc_{i+1} \rangle = - \dim \Ext^1(C_i, C_{i+1}) - \dim \Hom (C_{i+1}, C_i).\]
 	 and $\dim \Hom (C_i, C_{i+1}) = \dim \Ext^1 (C_{i+1}, C_i) = 0$.
 \end{description}
 
 From the identification between theta functions and the Caldero-Chapoton formula by equation \ref{eqn:newcc}, 
 the setup above is saying that we are considering the subrepresentations $E$ of dimension $e$ in $D$. 
 By the bendings of $\gamma$, we have 
 $e = \sum_i \lambda_i\bc_i$ for some $\lambda_i \in \N$.

\subsection{First Bending} \label{sec:firstbending}
As $\gamma$ has good bending over the first wall $\frakd_1$, we have
\[
	\langle \calE (d),\bc_1 \rangle > 0.
\]
As $C_i$ and $D$ are indecomposable, by lemmas in Section \ref{rmk:ext_zero}, we have
\[
	 \langle \calE (d),\bc_1 \rangle  = \chi (C_1, D) = \dim \Hom (C_1, D).
\]
We can describe the first bending as follows.

\begin{theorem} \label{thm:firstbending}
Let $\gamma$ be a broken line as defined above. 
 Consider the first bending of $\gamma$ over a general point $\stab_1$ on the wall $\frakd_1$ which corresponds to pre-projective/ pre-injective indecomposable representation $C_1$ of dimension vector $\textbf{c}_1$.
 Then the Hall algebra wall crossing automorphism is
\[ \Phi_{\frakD} (\frakd_1) z^{-(\calE (d),0)}= \sum_{\lambda} \calG^1_{\lambda_1,\bc_1}(D) z^{-(\calE (d),0)},\]
where $\calG^1_{\lambda,\bc_1} (D) = \{ \psi : V \raw D | V $ is a representation of dimension vector $\lambda\bc_1$, and $\ker \psi$ contains no subrepresentation of dimension vector proportional to $\textbf{c}_1 \}$.
\end{theorem}

\begin{proof}
By definition, the wall crossing at $\stab_1$ is
\begin{align*}
 \Phi_{\frakD} (\frakd_1) z^{-(\calE (d),0)} 
=& 1_{ss}(\stab_1) ^{-1} z^{-(\calE (d),0)} 1_{ss}(\stab_1) 
\end{align*}
Then we employ \eqref{eqn:1ss}
\begin{align*}
=&  1_{ss}(\stab_1) ^{-1} z^{-(\calE (d),0)} (1+\sum_{\ell \geq 1} [BGL_{\ell}(C_1) \raw \calM]) \\
=&  1_{ss}(\stab_1) ^{-1} (1+ \sum_{\ell \geq 1} q^{\ell \cdot \chi(\bc_1 , D)} [BGL_{\ell}(C_1) \raw \calM] ) z^{-(\calE (d),0)}\numberthis \label{eqn:commmute}\\
=&  1_{ss}(\stab_1) ^{-1} (1+ \sum_{\ell \geq 1} q^{\ell \dim \Hom(\bc_1 , D)} [BGL_{\ell}(C_1) \raw \calM] ) z^{-(\calE (d),0)} \\
=&  1_{ss}(\stab_1) ^{-1} \sum_{\ell \geq 0} 1_{ss}^D(\ell, \stab_1) z^{-(\calE (d),0)},\numberthis  \label{eqn:affinespace}
\end{align*}
where $1_{C_1}^D(\ell ,\stab_1)$ is the moduli space of semistable representations $C_1 '$ with a map $C_1 ' \raw D$, where $C_1 '$ has dimension vector $\ell\bc_1$.
 \eqref{eqn:commmute} follows from the commutativity relation in \eqref{eqn:comm}.
 We obtain \eqref{eqn:affinespace} by visualizing $q^{\ell \cdot \dim \Hom(\bc_1 , D)} $ as a vector bundle of rank $= \ell  \dim \Hom(\bc_1 , D)$.

We now need to understand $1_{ss}^D(\ell ,\stab_1)( \Spec \C)$.

Given $\psi: C_1 ' \raw D$,
 we consider $\ker \psi$.
 Take $\stab'$ to be a point very close to $\frakd_1$ in the positive $\textbf{c}_1$ direction.
 By Definition \ref{def:torsion} and Theorem \ref{thm:torsionpair}, we have
 \[ 0 \longrightarrow T \longrightarrow \ker \psi \longrightarrow F \longrightarrow 0 \]
where $T \in \calT (\stab')$, $F \in \calF (\stab')$.
 We claim $T $ has dimension vector proportional to $\dim C_1$.
 Firstly, as $T \subseteq \ker \psi \subseteq C_1'$, we have $\stab_1 (T) \leq 0$ since $C_1'$ is $\stab_1$-semistable.
 And by definition of $\calT$, $\stab' (T) >0$.
 As $\stab'$ can be arbitrary close to the wall $\frakd_1$, we have $\stab_1 (T) =0$. 

 Consider the quotient $C_1 = C_1' / T$.
 As both $C_1' $ and $T$ have dimension vectors as multiplies of $\textbf{c}_1$,
 $C_1 = \lambda\bc_1$ for some $\lambda$.
 Furthermore, $\psi$ induces a map $C_1 \raw D$ with kernel in $\calF(\stab')$.
 Thus we have the diagram
\[ \begin{tikzcd}
 0 \arrow[r] & T \arrow[r] &  C_1 ' \arrow[r] \arrow[d] &    C_1 \arrow[r] \arrow[dl] & 0 \\
             &  & D &   &    
\end{tikzcd}\]
where the row is an exact sequence.

From the exact sequence, we have
\[ 1_{ss}^D (\ell, \stab_1) = \sum_{h}  [BGL_{\ell}(C_1)  \raw \calM] * [\calG^1_{
\ell -h,\bc_1} (D) \raw \calM],  \]
where $\calG^1_{\ell-h,\bc_1}$ denotes denotes the stack of representations $C_1$ with dimension vector $(\ell -h )\bc_1$ and there exist a homomorphism $C_1 \raw D$ with kernel in $\calF(\stab')$.
 Then

\begin{align*}
 \sum_{\ell} 1_{ss}^D (\ell, \stab_1)
=&\sum_{\ell} \sum_{h}  [BGL_h(C_1) \raw \calM]  * 
	[\calG^1_{\ell-h,\bc_1} (D) \raw \calM] \\
=& \sum_{h}[BGL_h (C_1) \raw \calM]   * 
	\sum_{\ell} [\calG^1_{\ell-h,\bc_1} (D) \raw \calM] \\
=& 1_{ss} (\stab_1) *  \sum_{\lambda} [\calG^1_{\lambda,\bc_1} (D) \raw \calM] . 
\end{align*}
Therefore, we have
\[ \Phi_{\frakD} (\frakd_1) z^{-(\calE (d),0)} = 
	 1_{ss}(\stab_1) ^{-1} \sum_{\ell \geq 0} 1_{ss}^D(\ell, \stab_1) z^{-(\calE (d),0)}
	=  \sum_{\lambda } [\calG^1_{\lambda,\bc_1} (D) \raw \calM] z^{-(\calE (d),0)}.
\]

If $\textbf{c}_1$ is not regular, then by applying $\chi$ in \ref{eqn:chi1}, we have
\begin{align*}
	\chi ( \Phi_{\frakD} (\frakd_1) z^{-(\calE (d),0)}) 
	&=  z^{-(\calE (d),0)} \sum_{\lambda} \chi ( [\calG^1_{\lambda,\bc_1} (D) \raw \calM]) \\
	& =  z^{-(\calE (d),0)} \sum_{\lambda} \chi \left( \Gr( \lambda , \Hom (\textbf{c}_1, D)) \right) z^{\lambda_1 (p^*(\lambda\bc_1),\lambda_1\bc_1)}  
\end{align*}
This is the path-ordered product defined in \ref{eqn:pathordered}.

 Now we are ready to move forward to the second bending.
 As $\textbf{c} = \sum \lambda_i\bc_i$, we will take the term
\[
	[\calG^1_{\lambda_1,\bc_1} (D) \raw \calM]z^{-(\calE (d),0)}. 
\]
Applying the integration map, we have 
\[
	\chi([\calG^1_{\lambda_1} (D) \raw \calM]) z^{-(\calE (d),0)} = \chi \left( \Gr( \lambda_1 , \Hom (\textbf{c}_1, D)) \right) z^{-(\calE (d),0)+ \lambda_1(p^*(\bc_1),\bc_1)}.
\]

Let us consider the $q$-binomial coefficients $\dbinom{a}{b}_q = \frac{(q^a-1) \cdots (q^{a-b+1}-1)}{(q^b-1) \cdots (q-1)}$.
 Then the above equation is
 \[
	\chi([\calG^1_{\lambda_1} (D) \raw \calM]) z^{-(\calE (d),0)} = \dbinom{\dim \Hom (\textbf{c}_1, D)}{\lambda_1}_q z^{-(\calE (d),0)+ (p^*(\lambda_1 \bc_1),\lambda_1 \bc_1)}.
\]
Taking the limit $q \raw 1$ will give us the usual broken line attaching monomial $\dbinom{\langle \calE(d),\bc_1 \rangle}{\lambda_1} z^{-(\calE (d),0)+ (p^*(\lambda_1 \bc_1),\lambda_1 \bc_1)}$. 

Note that after the first bending, if we take $V_1 = C_1^{\oplus \lambda_1}$, we have
\[
0 \subset V_1
\]
as the first step of a filtration of $E \subset D$.
\end{proof}
\subsection{Second bending} \label{sec:2bending}
In this section and the next section, we will prove Theorem \ref{thm:morebending}. We will first state the statement for the second bending:

\begin{theorem} \label{thm:recapsecondbend}
Assume now $\gamma$ is taking the second bending from $\frakd_1^+$ to $\frakd_2$.
From the first bending, we obtain the filtration
 $
0 \subset V_1,
$
 where $V_1 = C_1^{\oplus \lambda_1}$.
Then the Hall algebra monomial associated to the second linear piece of the broken line would be
\[
[\calG^2_{s,\bc_2} (D) \raw \calM] z^{-(\calE (d),0)},
\]
where $[\calG^2_{s,\bc_2} (D) \raw \calM]$ 
is space having the Poincar\'{e} polynomial as
 \[[\Aff^{\lambda_2 \dim \Ext^1 (C_1^{\oplus \lambda_1})} \times \Gr(\lambda_2, \Hom(C_2, D/ C_1^{\oplus \lambda_1}) - \Ext^1 (C_1^{\oplus \lambda_1})) 
].
 \]
 After the second bending, 
 we obtain the filtration 
\begin{equation*} 
	 0 \subset V_1 \subset V_2, 
\end{equation*}
where $V_2 / V_{1} = C_2^{\oplus \lambda_2}$.
\end{theorem}

\begin{proof}
Now we consider $\gamma$ crosses a second wall $\frakd_2$.
 By Definition \ref{hall_brokendef} of Hall algebra broken lines, $[\calG^1_{\lambda_1,\bc_1} (D) \raw \calM]z^{-(\calE (d),0)}$ is attached to the linear piece of $\gamma$ after the first bending.
 In the usual broken line, the attached monomial is $\chi \left( \Gr( \lambda_1 , \Hom (\textbf{c}_1, D)) \right) z^{-(\calE (d),0)+ \lambda (p^*(\textbf{c}_1), \bc_1)}$.
 As we are assuming the crossing of $\gamma$ over $\frakd_2$ is good, we have
 \[
 \langle -\calE (D)+ p^*(\textbf{c}_1), -\textbf{c}_2 \rangle \geq 0,
 \]
 i.e.
 \[
  \calE(D)(\textbf{c}_2) - \{\textbf{c}_1,\bc_2\} \geq 0.
 \]
 From our assumption that $\gamma$ goes from $\frakd_1^+$ to $\frakd_2$ and the description in the beginning of Section \ref{sec:strata},
 we have 
 \[
  - \{\bc_1,\bc_2 \} \leq 0.
 \]
 This implies $\calE(D)(\textbf{c}_2) \geq 0$. Therefore 
 \[\calE(D)(\textbf{c}_2) = \chi (\textbf{c}_2, d) = \dim \Hom (\textbf{c}_2, D).\]

We can then apply the wall crossing for the second bending at a general point $\stab_2$ on $\frakd_2$  on $[\calG^1_{\lambda_1,\bc_1} (D) \raw \calM]z^{-(\calE (d),0)} $.
\begin{align*}
	&\Phi_{\frakD} (\frakd_2) \left( [\calG^1_{\lambda_1,\bc_1} (D) \raw \calM]  z^{-(\calE (d),0)} \right)\\
	 =& 1_{ss}(\stab_2) ^{-1}  \star [\calG^1_{\lambda_1,\bc_1} (D) \raw \calM] z^{-(\calE (d),0)} \star 1_{ss}(\stab_2) \\
	=& 1_{ss}(\stab_2) ^{-1} \star  \sum_{\ell} [\calG^1_{\lambda_1,\bc_1} (D) \raw \calM]
		\star q^{\dim (\Hom (\textbf{c}_2, D))}  1_{ss}(\ell, \stab_2) z^{-(\calE (d),0)} \\
	=& 1_{ss}(\stab_2) ^{-1}  \sum_{\ell} [\calG^1_{\lambda_1,\bc_1} (D) \raw \calM]
	 	\star 1^D_{ss} (\ell, \stab_2) z^{-(\calE (d),0)}
\end{align*}
where 
$1_{ss}^D(\ell, \stab_2)$ is the moduli space of semistable representation $\textbf{c}_2 '$ with a map $\textbf{c}_2 ' \raw D$, where $\textbf{c}_2 '$ is a representation with dimension vector $\ell\bc_2$. 

We will now put \eqref{eqn:inverse} into the above calculation, i.e. we need to investigate 
\begin{equation} \label{eqn:second}
 \sum_{\ell} \left(1 + \sum_k (-1)^k \prod_{l=1}^k [BGL_{r_l}(\textbf{c}_2) \raw \calM] \right) \star   [\calG^1_{\lambda_1,\bc_1} (D) \raw \calM]  \star 1^D_{ss} (\ell \stab_2)
\end{equation}
 \eqref{eqn:second} looks complicated. Actually it is very nice at each degree.
 We will first fix $\ell$ and $k$. 
 Without loss of generality, we write $\calG^1$ as $[\calG^1_{\lambda_1,\bc_1} (D) \raw \calM]$. 
 Let us look at the terms of dimension vector $s \bc_2$ in \eqref{eqn:second}.
 In order to attain $s\textbf{c}_2$, we can have two choices
 \begin{enumerate}
 	\item $B_k \star \calG^1 \star 1$ \label{front}
 	\item $B_{k-1} \star \calG^1 \star B'$, \label{back}
 \end{enumerate}
 where $B_k$ denote multiplying $k$ many classifying spaces $[BGL_{r_l} \raw \calM]$, for some $r_1, \dots, r_k \in \N$ such that $r_1 + \dots + r_k =s$,
  and
 $B'$ is an object in $1^D_{ss} (\ell, \stab_2)$. 
 Note that the inverse $1_{ss}(\stab_2) ^{-1}$ is an alternating sum.
 So elements in the forms of (\ref{front}) and (\ref{back}) differ by a sign.
 Therefore, we can group the summands for $s\textbf{c}_2$ into pairs
 \begin{equation} \label{eqn:degreek}
 \pm \left( B_k \star \calG^1 - B_{k-1} \star \calG^1 \star B' \right),  
 \end{equation}
 where $B_k = \prod_{l=1}^k [BGL_{r_l}(\textbf{c}_2) \raw \calM]$,
 $B_{k-1} = \prod_{l=1}^{k-1} [BGL_{r_l} (\textbf{c}_2) \raw \calM]$
 and
 $B'$ is the last multiple of $B_k$: $[BGL_{r_k} (\textbf{c}_2) \raw \calM]$.
 
 First consider
 \[ B_k \star \calG^1= \prod_{l=1}^k [BGL_{r_l}(\textbf{c}_2) \raw \calM] \star \calG^1 .\]

The product in $B_k$ means that we partition $s$ into $(r_1, \dots, r_k)$.
 Therefore, we are looking at the diagram
 \begin{equation}
  \begin{tikzcd}
 0 \arrow[r] &
\bc_2^{\oplus r_1} \oplus \cdots \oplus\bc_2^{\oplus r_k} \arrow[r] &
 	Y \arrow[r] &
 	C_1^{\lambda_1} \arrow[r] \arrow[d] &0\\
 & & & D & 
 \end{tikzcd}
 \end{equation}
with $Y$ such that the row is exact.

As the exact sequence may not split, the term $Y$ is actually 
 \begin{equation} \label{eqn:extterm}
 \left[ \frac{Ext^1 (C_1^{\oplus \lambda_1},\bc_2^{\oplus r_1} \oplus \cdots \oplus\bc_2^{\oplus r_k})}{G}
 \raw \calM
 \right]
 \end{equation}
 where $G$ is the group of matrices of the form
 $\left(
 \begin{array}{c c c c}
 \GL_{r_1} & * & * & * \\
 0 & \GL_{r_2} & * & * \\
 0 & 0 & \ddots & * \\
 0 & 0 & 0 & \GL_{r_k}
 \end{array}
 \right)$
  as noted at the end of Section \ref{sec:prop_hall}.
  
Next we consider the second term $B_{k-1} \star \calG^1 \star B' $ in \eqref{eqn:degreek}.
We will break down into two steps.
As the multiplication is associative, $B_{k-1} \star \calG^1 \star B' = B_{k-1} \star (\calG^1 \star B')$. 
Therefore, we will first look at $\calG^1 \star B'$ which leads us to \eqref{diag:1} while the whole term $B_{k-1} \star (\calG^1 \star B')$ will be associated to \eqref{diag:2}.
First, $\calG^1 \star B'$ gives
 \begin{equation} \label{diag:1}
  \begin{tikzcd}
 0 \arrow[r] &
 C_1^{\lambda_1} \arrow[r]\arrow[d]  &
 	U_1 \arrow[r] &
 	\textbf{c}_2^{r_k} \arrow[d] \arrow[r]&0\\
 &D & &D  & 
 \end{tikzcd}
 \end{equation}
for some $U_1$.
 We are looking at fibres of projection to $\calG^1$. Therefore we fix a map $C_1^{\lambda_1} \raw D$ in $\calG^1$.
 
 From the discussion in the beginning of Section \ref{sec:strata},
 we have $\Ext^1 (\textbf{c}_2, C_1) = 0$. Thus the exact sequence above is split. 
 Therefore the middle term is $U_1 = C_1^{\lambda} \oplus\bc_2^{r_k}$.
 Notice that there is an automorphism of the exact sequence, namely 
 \[
   \begin{tikzcd}
 0 \arrow[r] &
 C_1^{\lambda} \arrow[r] \arrow[d, equal] & C_1^{\lambda} \oplus\bc_2^{r_k} \arrow[r, "g"] \arrow[d, "\pi"]&\bc_2^{r_k} \arrow[dll] \arrow[d, equal] \arrow[r]&0\\
 0 \arrow[r] &C_1^{\lambda} \arrow[r, "i"] & C_1^{\lambda} \oplus\bc_2^{r_k} \arrow[r] &\bc_2^{r_k} \arrow[r] &0
 \end{tikzcd}
 \]
 with $g \circ \pi \circ i$ is an identity.
 We obtain the object
\begin{equation} \label{eqn:firstterm}
 \left[
 \frac{\Hom(C_2^{r_k}, D) / \Hom(C_2^{r_k}, C_1^{\oplus \lambda_1})}{\GL_{r_k}}
 \raw \calM \right]
 =
 \left[
 \frac{\Hom(\textbf{c}_2^{r_k}, D/ C_1^{\oplus \lambda_1})}{\GL_{r_k}}
 \raw \calM \right].
\end{equation} 

The second multiplication $B_{k-1} \star ( \calG^1 \star B')$ gives 
 \begin{equation} \label{diag:2}
  \begin{tikzcd}
 0 \arrow[r] &
 C_2^{\oplus r_1} \oplus \cdots \oplus C_2^{\oplus r_{k-1}} \arrow[r] &
 	U_2  \arrow[r] & C_1^{\lambda} \oplus C_2^{r_k} \arrow[r] \arrow[d] &0 \\
  & & & D & 
 \end{tikzcd}
 \end{equation}
 These exact sequences are classified by $\Ext^1(C_1^{\lambda} \oplus C_2^{r_k},C_2^{\oplus r_1} \oplus \cdots \oplus C_2^{\oplus r_{k-1}} ) = \Ext^1(C_1^{\lambda},C_2^{\oplus l_1} \oplus \cdots \oplus C_2^{\oplus l_{r-1}} )$ as $C_2$ is non-regular.
 Combining with \eqref{eqn:firstterm} the product gives us
 \begin{equation} \label{eqn:secondterm}
 \left[ \frac{\Hom(C_2^{r_k}, D/ C_1^{\oplus \lambda_1}) 
 \times \Ext^1(C_1^{\lambda_1},C_2^{\oplus r_1} \oplus \cdots \oplus C_2^{\oplus r_{k-1}} )}{T} 
 \raw \calM \right],
 \end{equation}
 where $T$ is as in \eqref{eqn:extterm}.

 The value of \eqref{eqn:degreek} would be the difference between \eqref{eqn:extterm} and  \eqref{eqn:secondterm}, i.e.
 \begin{align*}
  &\pm \Bigg( \left[ \frac{\Ext^1 (C_1^{\oplus \lambda_1}, C_2^{\oplus r_1} \oplus \cdots \oplus C_2^{\oplus r_k})}{T}
 \raw \calM
 \right]
 - \\
  &\left[ \frac{\Hom(C_2^{r_k}, D/ C_1^{\oplus \lambda_1}) 
 \times \Ext^1(C_1^{\lambda_1},C_2^{\oplus r_1} \oplus \cdots \oplus C_2^{\oplus r_{k-1}} )}{T} 
 \raw \calM \right] \Bigg)
 \end{align*}

Let $\eta = \dim \Hom(C_2, D/ C_1^{\oplus \lambda_1})$ and $\gamma = \dim \Ext^1 (C_1^{\oplus \lambda_1}, C_2)$.
 Recall $s = l_1 + \cdots + l_r $ as we are counting at degree $s$.
 Applying $\chi$ to the term above gives us
 \begin{align*}
   &\pm \frac{q^{s \gamma}-q^{(r_1 + \cdots + r_{k-1}) \gamma+r_k h} }{\left( \prod_{l=1}^k q^{r_l(r_l-1)/2} \prod_{t=1}^{r_l} (q^t-1) \right)  \prod_{u<v} q^{r_u r_v}} \\
   =& \pm \frac{q^{s \gamma} (q^{r_k \eta- r_k \gamma} - 1)  }{\left( \prod_{l=1}^k q^{r_l(r_l-1)/2} \prod_{t=1}^{r_l} (q^t-1) \right)  \prod_{u<v} q^{r_u r_v}}.
 \end{align*}

One can show that after summing over all the partitions of $s$,
 the multiplities $s$ term of \eqref{eqn:second} after applying $\chi$ would be
 \[q^{s \gamma} \dbinom{\eta-\gamma}{\lambda_2}_q.\]
 
Let $S^2$ be the space of the second bending with fixing $C_1^{\lambda_1} \raw D$.
 Then $S^2$ would have the same Poincar\'{e} polynomial as the following space:
\[
\sum_{s} [\Aff^{s \gamma} \times \Gr(s, \Hom(C_2, D/ C_1^{\oplus \lambda_1}) - \Ext^1 (C_1^{\oplus \lambda_1}, C_2))],
\]
where if $V$, $W$ are vector space, we use the relation $V-W$ to denote the vector space of dimension $\dim V - \dim W$.
 We now take the map $C_1^{\lambda_1} \raw D$ back into account,
 and denote $[\calG^2_{s,\textbf{c}_2} \raw \calM ]$ as the degree $s$ for the space we obtained after the second bending.
 Then $[\calG^2_{s,\textbf{c}_2} \raw \calM ]$ have the same Poincar\'{e} polynomial as
\[
[\Aff^{s \dim \Ext^1 (C_1^{\oplus \lambda_1})} \times \Gr(s, \Hom(C_2, D/ C_1^{\oplus \lambda_1}) - \Ext^1 (C_1^{\oplus \lambda_1})) 
\times
\Gr(\lambda_1, \Hom (C_1, D))
].
\]

Therefore, the Hall algebra wall crossing at second bending would then be
\begin{align*}
 \Phi_{\frakD} (\frakd_2) \left( [\calG^1_{\lambda_1,\bc_1} (D) \raw \calM]  z^{-(\calE (d),0)} \right)
=\sum_{s} [\calG^2_{s,\bc_2} (D) \raw \calM] z^{-(\calE (d),0)}.
\end{align*}

From the construction, we take $\lambda_2$ pieces of $\textbf{c}_2$.
 Thus the attaching Hall algebra monomial is
 \[
 [\calG^2_{\lambda_2,\bc_2} (D) \raw \calM] z^{-(\calE (d),0)}.
 \]
 Applying integration map will give us
 \[
 q^{\lambda_2 \gamma} \dbinom{\chi(\textbf{c}_2, d-\lambda_1\bc_1)}{\lambda_2}_q \dbinom{\chi(\textbf{c}_1, d)}{\lambda_1}_q
 z^{-\calE (D)+ (p^*(\lambda_1\bc_1 + \lambda_2\bc_2),(\lambda_1\bc_1 + \lambda_2\bc_2)}.
 \]
 Taking the limit $q \raw 1$, will give us the usual attaching monomial.

Moreover, in the Hall algebra multiplication, we build up the representation $V_2$ of dimension vector $\lambda_1\bc_1 + \lambda_2\bc_2$ from the exact sequence.
 Thus, we have a filtration
\[
0 \subset V_1 \subset V_2
\]
with the quotient $V_2/ V_1 = C_1^{\oplus \lambda_1}$.
\end{proof}

\subsection{The \texorpdfstring{$k$}--th bending} \label{sec:kbend}

The bendings of the broken line after the second bending would behave similarly. 
Thus we can repeat the last section inductively.
Assume now $\gamma$ is having the $j$-th bending from $\frakd_{j-1}^+$ to $\frakd_j$.
 From the previous $j-1$ bendings, we have the  Hall algebra monomial 
 \[[\calG^{j-1}_{\lambda_1 , \cdots , \lambda_{j-1},\bc_{j-1}} \raw \calM] z^{-(\calE (d),0)}.\]
 And there is the filtration
 \[
 0 \subset V_1 \subset \cdots \subset V_{j-1}.
 \] 
 
Repeat the calculation in Section \ref{sec:2bending} for bending over $\frakd_j$ and taking $\lambda_j$ copies of $C_j$, we will then have the space $S^j$ having the same Poincar\'{e} polynomial as
 \[ [\Aff^{\lambda_j \Ext^1 (V_{j-1}, C_j)} \times
 \Gr(\lambda_j, \Hom(C_j, D/ V_{j-1} )   - \Ext^1 (V_{j-1}, C_j))].
 \]

This also builds up one more piece in the filtration
\[
0 \subset V_1 \subset \cdots \subset V_j,
\]
where $V_j / V_{j-1} = C_j^{\oplus \lambda_j}$.
After the last bending of $\gamma$, notice that in our set up, we have
\[
	\textbf{c} = \sum_{j=1}^s \lambda_j\bc_j.
\]
Thus the Hall algebra multiplication allows us to build up a filtration for a subrepresentation $C$ of $D$ with dimension vector $\textbf{c}$.
\begin{equation} \label{eqn:filtration}
	 0 \subset V_1 \subset V_2 \subset \cdots \subset V_s =E, 
\end{equation}
where $V_j / V_{j-1} = C_j^{\oplus \lambda_j}$.

This finished proofing Theorem \ref{thm:morebending}:

\begin{theorem} \label{thm:kbend}
If $\gamma$ is taking the $j$-th bending from $\frakd_{j-1}^+$ to $\frakd_j$.
From the $j-1$-th bending, we obtain the filtration
\[
 0 \subset V_1 \subset \cdots \subset V_{j-1},
 \]
 with the quotient $V_{\ell}/ V_{\ell-1} = C_{\ell-1}^{\oplus \lambda_{\ell-1}}$.
Then the Hall algebra monomial associated to the $j$-th linear piece of the broken line would be
\[[\calG^{j-1}_{\lambda_1 , \cdots , \lambda_{j-1},\bc_{j-1}} \raw \calM] z^{-(\calE (d),0)},\]
where $[\calG^{j-1}_{\lambda_1 , \cdots , \lambda_{j-1},\bc_{j-1}} \raw \calM]$ 
is space having the Poincar\'{e} polynomial as
 \[ [\Aff^{\lambda_j \Ext^1 (V_{j-1}, C_j)} \times
 \Gr(\lambda_j, \Hom(C_j, D/ V_{j-1} )   - \Ext^1 (V_{j-1}, C_j))].
 \]
 At the last bending, 
 we obtain filtration for a subrepresentation $C$ of $D$ with dimension vector $\textbf{c}$.
\begin{equation*} 
	 0 \subset V_1 \subset V_2 \subset \cdots \subset V_s =E, 
\end{equation*}
where $V_j / V_{j-1} = C_j^{\oplus \lambda_j}$.
\end{theorem}

\subsection{Harder-Narasimhan filtration} \label{sec:HN}

In two dimensions, a broken line $\gamma$ can actually describe a Harder-Narasimhan filtration for a subrepresentation $E \subset D$.

Let $\enpt$ be in the positive chamber. Now consider the stability function $Z: K(\rep (Q)) \raw \C$ as
\[
	Z (F) = (\calE (\textbf{d}) - p^*(\textbf{c})) \cdot \textbf{f} + i \enpt \cdot \textbf{f}.
\]
Now consider the line $\beta: (0,1) \raw \frakD$ starting from $-\calE (d) - p^*(e))$ and ending at $\enpt$. $\beta$ is parametrized by 
 $\beta(t) = (1-t) (-\calE (\textbf{d})+ p^*(\textbf{c})) +t \enpt$.
 When $\beta$ cross the walls $\frakd$ with normal vector $f \in N^+$ at $t$,
 we have
 \[
 \langle (1-t) (-\calE (\textbf{d}) + p^*(\textbf{c})) +t \enpt, \textbf{f} \rangle = 0. 
 \]
 Solving the above will give
 \[
 t^{-1} -1 = \frac{\enpt \cdot \textbf{f}}{(\calE (\textbf{d}) - p^*(\textbf{c}))(\textbf{f})}.
 \]
This is the slope of $Z$.
 
Let $\gamma$ be the broken line we considered in the beginning of Section \ref{sec:strata}.
 If $\gamma$ has good bending over $\frakd$, then $-(\calE (d) - p^*(\textbf{c}))(-\textbf{f}) \geq 0$ from the properties of $\gamma$.
 Thus $(\calE (\textbf{d}) - p^*(\textbf{c}))(\textbf{f}) \geq 0$.
 Therefore, $\frac{\enpt \cdot \textbf{f}}{(\calE (\textbf{d}) - p^*(\textbf{c}))(\textbf{f})}$ has the same ordering as $\phi (F) = \frac{1}{\pi} \arg Z(F)$.
 
 As $t^{-1} -1$ is strictly decreasing, $\phi(F)$ is decreasing along $\beta$.
 Therefore slopes of the quotients in the filtration \eqref{eqn:filtration} are strictly decreasing.
 Hence the filtration \eqref{eqn:filtration} is a Harder-Narasimhan filtration.
 
\subsection{Example}
 
Let us illustrate what we have in this section by an example. We will consider the Kronecker 2-quiver $Q_2$.
Since the $p^*$ map is injective in this case. We will simply consider the $\calA$ scattering diagram instead of the $\cAp$.
Take the indecomposable representation $\C^5 \rraws \C^6$. Then we will consider the subrepresentation $E$ with dimension vector $(2,4)$. By combining theta function and the Caldero-Chapoton formula, we learn that the Euler characteristic is $\chi (\C^2 \rraws \C^4, \C^5 \rraws \C^6) = 18$.

In terms of broken line language, we calculate
\[
-\calE(5,6) = (7, -6), \qquad -\calE(5,6)+p^*(2,4) = (-1,-2).
\]
Therefore, we are looking for broken line $\gamma$ with initial slope $(7, -6)$, final slope $(-1,-2)$ and endpoint at some point $\enpt$ in the positive chamber.

If we set $\enpt $ as a irrational point near $(2,1)$, we will obtain two broken lines $\gamma_1$ (blue) and $\gamma_2$ (red) as shown in Figure \ref{fig:24in56}.
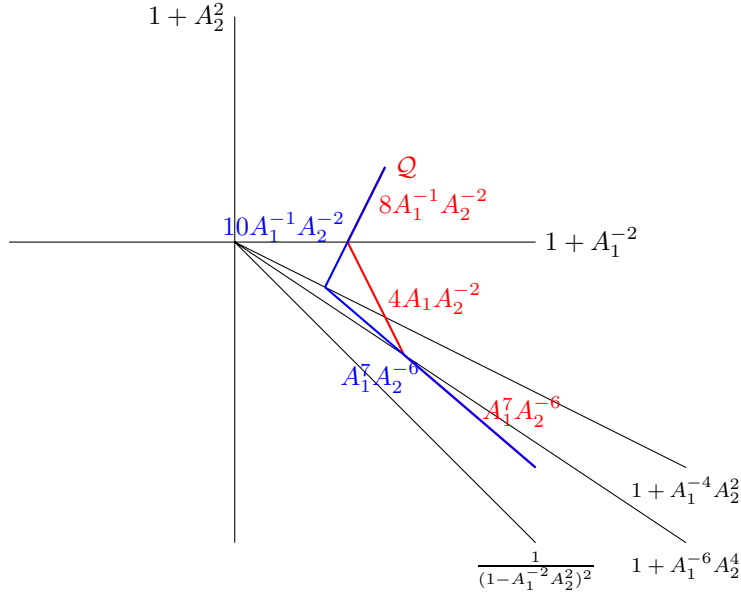
\begin{figure}
\centering
\begin{tikzpicture}
\draw
(-3,0) -- (4,0) node[right] {$1+A_1^{-2}$}
(0,-4) -- (0,3) node[left] {$1+A_2^2$}
(0,0) -- (4,-4) node[below] {\footnotesize{$\frac{1}{(1-A_1^{-2} A_2^2)^2}$}}
(0,0) -- (6,-4) node[below] {\footnotesize{$1+A_1^{-6}A_2^4$}}
(0,0) -- (6,-3) node[below] {\footnotesize{$1+A_1^{-4}A_2^2$}};

\draw[red, thick]
(1.5,0) --node[right=1pt] {$8 A_1^{-1}A_2^{-2}$} (2,1) node[right] {$\enpt$}
(9/4, -3/2) --node[right=1pt] {$4A_1 A_2^{-2}$} (1.5, 0)
(4, -3) --node[right=1pt] {$A_1^7 A_2^{-6}$}(9/4, -3/2);

\draw[blue, thick]
(6/5,-3/5) --node[left] {$10 A_1^{-1}A_2^{-2}$} (2,1) 
(4, -3) --node[left] {$A_1^7 A_2^{-6}$} (6/5,-3/5)
;
\end{tikzpicture}
\caption{$\C^2 \rightrightarrows \C^4 \subset \C^5 \rightrightarrows \C^6$}
\label{fig:24in56}
\end{figure}

Let us first start with $\gamma_1$ (the blue line).
 The broken line $\gamma_1$ bends over the wall $\frakd=\{ \R_{\geq 0} (2,-1), 1+z^{(-4,2)} \}$. From Section \ref{sec:firstbending}, we see that the bending gives two copies of $\C \rraws \C^2$. Therefore the filtration for $E_1$ would be
\[
0 \subset (\C \rraws \C^2)^2 = C_1.
\]
 Furthermore the Hall algebra element after bending is $\calG^1_{2, (1,2)}(\C^5 \rraws \C^6) \raw \calM$,
 where the objects are $\C^2 \rraws \C^4$ with a map $(\C^2 \rraws \C^4) \xrightarrow{\psi} (\C^5 \rraws \C^6)$ such that $\ker \psi$ does not contain any $\C \rraws \C^2$.
 Thus the space is 
 \[\Gr (2, \Hom \left( (\C \rraws \C^2), (\C^5 \rightrightarrows \C^6) \right).\]
 
Next, we have $\gamma_2$.
 The broken line $\gamma_2$ has two bendings: first at the wall
 $\frakd_1 = \{ \R_{\geq 0} (3,-2), 1+z^{(-6,4)} \}$ then at the wall
 $\frakd_2 = \{ \R(1,0), 1+z^{(-2,0)} \}$.
 From the calculation of slopes, we know that both bendings take one copy of $\C^2 \rraws \C^3$, $0 \rraws \C$ respectively.
 This implies the filtration
 \[
 0 \subset \C^2 \rraws \C^3 \subset \C^2 \rraws \C^4 =: E_2,
 \]
 showing that $E_2 = (\C^2 \rraws \C^3) \oplus (0 \rraws \C)$.
 
The first bending of $\gamma_2$ is similar to $\gamma_1$.
 Thus we obtain $[\calG^1_{1, (2,3) }(\C^5 \rraws \C^6)\raw \calM]$ by Proposition \ref{thm:firstbending}. The second bending over a general point $\stab_2 \in \frakd_2$ yields a term in
\[
 1_{ss}(\stab_2) ^{-1}  \star [\calG^1_{1, (2,3) }(\C^5 \rraws \C^6)\raw \calM] z^{(7,-6)} \star 1_{ss}(\stab_2).
\]
Write $\calG = [\calG^1_{1, (2,3) }(\C^5 \rraws \C^6)\raw \calM]$.
 As we only need degree 1 in this case, we only need to  consider
\begin{equation} \label{eqn:example}
 \calG z^{(7,-6)} \star [ GL_1(0 \rraws \C) \raw \calM] - [GL_1(0 \rraws \C) \raw \calM] \star \calG z^{(7,-6)}.
\end{equation}

After commuting $z^{(7,-6)} $ and $ BG_m(0 \rraws \C)$, the first term becomes
\begin{align*}
& \calG z^{(7,-6)} \star [BG_m(0 \rraws \C) \raw \calM] \\
= & \calG q^{\dim \Hom(0 \rraws \C, \C^5 \rraws \C^6)} [BG_m(0 \rraws \C) \raw \calM] z^{(7,-6)} \\
= & \calG \star 1^D_{ss} (1, \stab_2)  z^{(7,-6)}
\end{align*}

The product consists of diagrams
\begin{equation} \label{seq:first}
\begin{tikzcd}
0 \arrow[r] & ( \C^2 \rraws \C^3) \arrow[r] \arrow[d] &  (\C^2 \rraws \C^4) \arrow[r] \arrow[dl]& (0 \rraws \C) \arrow[r] & 0 \\
& (\C^5 \rraws \C^6) & & &
\end{tikzcd}
\end{equation}
where the map $(\C^2 \rraws \C^4) \raw (\C^5 \rraws \C^6)$ extends the map $(\C^2 \rraws \C^3) \raw (\C^5 \rraws \C^6)$.

Note that $\Ext^1(0 \rraws \C, \C^2 \rraws \C^3) =0$. Thus the term $\C^2 \rraws \C^4$ in the exact sequence above is actually $(0 \rraws \C) \oplus ( \C^2 \rraws \C^3)$.
 In Section \ref{sec:2bending}, we obtained the space
\[\left[\frac{\Hom(0 \rraws \C, \C^5 \rraws \C^6 / \C^2 \rraws \C^3)}{GL_1(0 \rraws \C)} \right]\]
after fixing $(\C^2 \rraws \C^3) \raw (\C^5 \rraws \C^6)$. 
 Applying the integration map, we have
 \[
 \chi(\left[\frac{\Hom(0 \rraws \C, \C^5 \rraws \C^6 / \C^2 \rraws \C^3)}{GL_1(0 \rraws \C)}\right]) = \frac{q^{\eta}}{q-1},
 \]
 where $\eta = \dim \Hom(0 \rraws \C, \C^5 \rraws \C^6 / \C^2 \rraws \C^3)$.
 
 The second term $ BGL_1(0 \rraws \C) \star \calG z^{(7,-6)}$ gives
 \begin{equation} \label{seq:second}
  \begin{tikzcd}
 0 \arrow[r] & (0 \rraws \C) \arrow[r]  &  (\C^2 \rraws \C^4) \arrow[r] & ( \C^2 \rraws \C^3) \arrow[r]\arrow[d] & 0 \\
&  & &(\C^5 \rraws \C^6) &
 \end{tikzcd}
 \end{equation}
 which yields $\left[\frac{Ext^1(\C^2 \rraws \C^3,0 \rraws \C )}{GL_1(0 \rraws \C )} \right]$ and\[\chi (\left[\frac{Ext^1(\C^2 \rraws \C^3,0 \rraws \C )}{BL_1(\C^5 \rraws \C^6 )} \right]) = \frac{q^\gamma}{q-1},\]
  where $\gamma = \dim Ext^1(\C^2 \rraws \C^3,0 \rraws \C )$.
 Combining together, we have
 \[\frac{q^{\eta}}{q-1} - \frac{q^\gamma}{q-1} = q^\gamma \frac{q^{\eta-\gamma}-1}{q-1} = q^\gamma \dbinom{\eta-\gamma}{1}_q.\]
 This is saying the moduli space $S^2$ has the same Poincar\'{e} polynomial as 
 \[\Aff^1 \times \left( \Gr\left(1, \Hom(0 \rraws \C, \C^5 \rraws \C^6 / \C^2 \rraws \C^3) - \Ext^1(\C^2 \rraws \C^3,0 \rraws \C )\right) \right).\]
 
 Furthermore, note that our stability function is
 \[ Z(F) = -(-1,-2) \cdot F + i (2,1) \cdot F\]
 
 By calculating, we have $Z(2,3) = 8 + 7i$ and $Z(0,1) = 2 + i$. Thus $\phi(F) = \frac{1}{\pi} \arg Z(F)$ is decreasing. Hence the filtration 
  \[
 0 \subset \C^2 \rraws \C^3 \subset \C^2 \rraws \C^4 
 \]
 is Harder-Narasimhan.

\bibliographystyle{alpha} 

\bibliography{biblio} 
\Addresses
\end{document}